\documentclass[oneside,a4paper,11pt,notitlepage]{amsart}
\usepackage[left=1.2in, right=1.2in, top=1.5in, bottom=1.5in]{geometry}
\usepackage{yfonts}
\usepackage{mathtools}
\usepackage{beton}
\newtheorem*{namedthm}{\namedthmname}
\newcounter{namedthm}

\makeatletter
\newenvironment{named}[1]
{\def\namedthmname{Theorem #1}%
	\refstepcounter{namedthm}%
	\namedthm\def\@currentlabel{#1}}
{\endnamedthm}
\makeatother
\usepackage{braket}
\usepackage{amsmath,amssymb,amsthm,graphicx,mathrsfs,url}
\usepackage{physics}
\usepackage[usenames,dvipsnames]{color}
\usepackage{fullpage}
\usepackage{hyperref}
\usepackage{stackrel}
\usepackage{amsthm}
\usepackage{enumitem}
\usepackage{stackrel}
\usepackage{comment}
\usepackage{verbatim, lmodern}
\usepackage{xcolor}
\usepackage{eurosym}
\newcommand\nb{\nabla}
\newcommand{\beq}{\begin{equation} \begin{split}}
\newcommand{\eeq}{\end{split} \end{equation}}

\newcommand\Omg{\Omega}
\newcommand\ext{{\rm ext}}

\renewcommand\and{\qquad\text{and}\qquad}

\newcommand\sm{\setminus}

\def\sfH{\mathsf{H}}

\def\bm1{\mathbbm{1}}

\def\s{\sigma}

\def\p{\partial}

\renewcommand{\iff}{\textit{if, and only if,}\,}

\def\Re{{\rm Re}\,}

\def\arr{\rightarrow}

\def\tt{\theta}
\def\aa{\alpha}
\def\lm{\lambda}

\def\s{\sigma}

\def\sess{\sigma_{\rm ess}}

\def\ii{{\mathsf{i}}}
\def\p{\partial}

\def\kp{\kappa}

\def\sfH{\mathsf{H}}

\def\dd{{\,\mathrm{d}}}

\def\sfU{\mathsf{U}}

\newcounter{counter_a}
\newenvironment{myenum}{\begin{list}{{\rm(\roman{counter_a})}}%
{\usecounter{counter_a}
\setlength{\itemsep}{1.ex}\setlength{\topsep}{0.8ex}
\setlength{\leftmargin}{5ex}\setlength{\labelwidth}{5ex}}}{\end{list}}

\newcommand{\eg}{{\it e.g.}\,}
\newcommand{\ie}{{\it i.e.}\,}
\newcommand{\cf}{{\it cf.}\,}

\numberwithin{figure}{section}
\numberwithin{equation}{section}
\theoremstyle{plain}
\newtheorem*{thm*}{Theorem}

\theoremstyle{remark}

\theoremstyle{plain}

\newcommand{\supp}{\mathrm{supp}\,}
\newcommand{\beu}{\begin{equation*}}
\newcommand{\eeu}{\end{equation*}}
\newcommand{\besu}{\begin{equation*}
\begin{aligned}}
\newcommand{\eesu}{\end{aligned}
\end{equation*}}
\newcommand{\bes}{\begin{equation}
\begin{aligned}}
\newcommand{\ees}{\end{aligned}
\end{equation}}

\newcommand\cB{\mathcal B}

\newcommand\cH{\mathcal H}

\newcommand\cL{\mathcal L}

\newcommand\cR{\mathcal R}

\newcommand\RR{\mathbb R}

\newcommand\frh{\mathfrak h}

\newcommand\ov{\overline}

\newcommand\void[1]{}

\def\ov{\overline}
\def\eps{\varepsilon}
\def\ran{{\rm ran\,}}

   \def\dH{{\mathbb H}}   
      
   \def\dN{{\mathbb N}}   
      \def\dR{{\mathbb R}}

   \def\sfH{{\mathsf H}}

      \def\sfU{{\mathsf U}}

   \def\cB{{\mathcal B}}   \def\cC{{\mathcal C}}
      
   \def\cH{{\mathcal H}}   
      \def\cL{{\mathcal L}}
      
      \def\cR{{\mathcal R}}

\newcommand{\dom}{\mathrm{dom}\,}

\hypersetup{colorlinks=true, linkcolor=blue,citecolor=citegreen}
\definecolor{citegreen}{rgb}{0.2,0.2,0.6}

\DeclareMathOperator{\sn}{sn}
\DeclareMathOperator{\cn}{cn}
\DeclareMathOperator{\cut}{Cut}
\newcommand{\scalar}[2]{\langle #1,#2 \rangle }

\newtheorem*{CorA}{Corollary}

\newtheorem{Theorem}{Theorem}[section]
\newtheorem{Lemma}[Theorem]{Lemma}
\newtheorem{Proposition}[Theorem]{Proposition}
\newtheorem{Corollary}[Theorem]{Corollary}

\theoremstyle{definition}
\newtheorem{definition}[Theorem]{Definition}
\newtheorem{Remark}[Theorem]{Remark}

\usepackage[normalem]{ulem}
\newcommand{\Hm}[1]{\leavevmode{\marginpar{\tiny%
			$\hbox to 0mm{\hspace*{-0.5mm}$\leftarrow$\hss}%
			\vcenter{\vrule depth 0.1mm height 0.1mm width \the\marginparwidth}%
			\hbox to
			0mm{\hss$\rightarrow$\hspace*{-0.5mm}}$\\
			\relax\raggedright #1}}}

\definecolor{DarkGreen}{rgb}{0,0.5,0.1}

\newcommand\soutD{\bgroup\markoverwith
	{\textcolor{DarkGreen}{\rule[.5ex]{2pt}{1pt}}}\ULon}

\title{Optimisation of the lowest Robin eigenvalue in 
exterior domains of the hyperbolic plane}

\author[A.~Celentano]{Antonio Celentano}
\address{(A.~Celentano) Dipartimento di Matematica e Applicazioni ``R. Caccioppoli'', Universit\`a degli studi di Napoli Federico II, Via Cintia, Complesso Universitario Monte S. Angelo, 80126 Napoli, Italy} 
\email{antonio.celentano2@unina.it}

\author[D.~Krej\v{c}i\v{r}\'{i}k]{David Krej\v{c}i\v{r}\'{i}k}
\address{(D.~Krej\v{c}i\v{r}\'{i}k) Department of Mathematics\\ Faculty of Nuclear Sciences and Physical
	Engineering\\
	Czech Technical University in Prague\\ Trojanova 13, 120 00, Prague, Czech
	Republic 
}
\email{david.krejcirik@fjfi.cvut.cz}
\author[V.~Lotoreichik]{Vladimir Lotoreichik}
\address{(V.~Lotoreichik)
	Department of Theoretical Physics\\
	Nuclear Physics Institute, Czech Academy of Sciences, 
	25068 \v{R}e\v{z}, Czech Republic
}
\email{lotoreichik@ujf.cas.cz}



\begin{document}
%

\begin{abstract}
We consider the Robin Laplacian in the exterior of a bounded
simply-connected Lipschitz domain in the hyperbolic plane. We show that the essential spectrum of this operator is $[\frac14,\infty)$ and that, under convexity assumption on the domain, there exist discrete eigenvalues below $\frac14$ if, and only if, the Robin parameter is below a non-positive critical constant, which depends on the shape of the domain. 
As the main result, we prove that the lowest Robin eigenvalue for the exterior of a bounded geodesically convex domain $\Omega$ in the hyperbolic plane does not exceed such an eigenvalue for the exterior of the geodesic disk, whose geodesic curvature of the boundary is not smaller than the averaged geodesic curvature of the boundary of $\Omega$. This result implies as a consequence that under fixed area or fixed perimeter constraints
the exterior of the geodesic disk maximises the lowest Robin eigenvalue among  exteriors of bounded geodesically convex domains. Moreover, we obtain under the same geometric constraints a reverse inequality between the critical constants.
\end{abstract}

\maketitle

\section{Introduction}
%
\subsection{Background and motivation}
Optimisation of eigenvalues of differential operators with respect to shape of the domain is a classical topic in spectral geometry with many results and challenging open problems; see \eg the monographs~\cite{H06, H17} and the references therein.
A prominent question in this context is to optimise the first eigenvalue of the Laplacian
with certain boundary conditions and under
suitable geometric constraints such as fixed volume or fixed perimeter.  The case of Dirichlet
boundary conditions is the most classical and goes back to the papers by Faber~\cite{F23} and Krahn~\cite{K25, K26}. For the Neumann Laplacian the lowest eigenvalue is always zero and one is then naturally interested in optimisation of the first non-zero eigenvalue~\cite{S54,W56}.
In both these paradigmatic examples,
in analogy with the classical isoperimetric inequality,
the ball turns out to be the optimal geometry.

In the recent years, Robin boundary conditions attracted a lot of attention. In this case, the sign of the boundary parameter plays a significant role. It was proved by Bossel~\cite{B86} 
in two dimensions and then later by Daners~\cite{D06} 
in all space dimensions that under fixed volume constraint the ball minimises the lowest Robin eigenvalue, when the boundary parameter is positive. Alternative proofs were proposed in 
\cite{Alvino-Nitsch-Trombetti,Bucur-Giacomini_2010,Bucur-Giacomini_2015}.
Hence, despite the necessity to deal with
significant methodological challenges,
the positive Robin case brings no surprises:
the ball remains the optimal geometry.

The situation changes drastically with the negative boundary parameter. It was proved by Freitas and the second author of the present paper~\cite{FK15} that under fixed volume constraint the ball is not an optimiser in general, but is a maximiser in two dimensions if the negative Robin parameter exceeds a critical negative constant, which only depends on the area of the domain. On the other hand, under fixed perimeter constraint in two dimensions~\cite{AFK17} and under fixed area of the boundary with additional convexity assumption in higher dimensions~\cite{BFNT19}, it is proved that the ball is the maximiser of the lowest eigenvalue for all negative boundary parameters. Moreover, it is conjectured in~\cite{AFK17} that under fixed volume constraint the ball is a maximiser of the lowest Robin eigenvalue with negative boundary parameter among simply-connected planar domains and among convex domains in higher space dimensions.

Optimisation of the lowest Robin eigenvalue was studied also for bounded domains in manifolds~\cite{ACCNT, CCL22, KhL22, S20} and in the presence of a homogeneous magnetic field~\cite{KaL22}. The paper~\cite{KhL22} provides, in particular, optimisation results on bounded domains in the hyperbolic plane or a more general negatively curved manifold. However, to the best of our knowledge, it remains  an open problem whether among all bounded smooth simply-connected domains in the hyperbolic plane with fixed perimeter the lowest eigenvalue of the Robin Laplacian with negative boundary parameter is maximised by the geodesic disk. 

A lot of attention has also been paid to optimisation of the second and the third Robin eigenvalues of the Laplacian on a bounded domain in Euclidean space~\cite{FL20, FL21,GL21, KL25} and in a manifold~\cite{LL25}. The latter paper deals, in particular, with optimisation of the second Robin eigenvalue on a bounded domain in the hyperbolic plane yielding optimality of the geodesic disk under suitable geometric constraints and for the boundary parameter from a certain natural range.

Along with optimisation of Laplace eigenvalues on general bounded domains, there is an interest in optimisation of such eigenvalues on doubly-connected domains; see~\cite{H63, PW61} for classical works and \cite{Anoop-Bobkov-Drabek, ABG25} for more recent refinements. 

As a sort of extremal case of a doubly-connected domain, the second and the third authors of the present paper initiated in the series of papers~\cite{KL18, KL20, KL24} the investigation
of eigenvalue optimisation for the Robin Laplacian on exterior domains with negative boundary parameters. Here the main challenge is the presence of the essential spectrum due to the unboundedness of the geometry in question. In the planar case, it was proved in~\cite{KL18, KL20} that the essential spectrum of the Robin Laplacian with a negative boundary parameter on the complement $\Omg^\ext :=\dR^2\sm\ov\Omg$ of a bounded simply-connected smooth domain $\Omega\subset\dR^2$ coincides with the non-negative semi-axis and that this operator has at least one negative discrete eigenvalue, which is maximised by the exterior of the disk provided that the area or the perimeter of~$\Omega$ is fixed. The case of exterior domains in higher dimensions 
is more complicated~\cite{KL20}, but balls are still at least locally optimal~\cite{B25}.

A dual problem to the Robin Laplacian on an exterior domain is the Steklov eigenvalue problem on an exterior domain, which has recently been studied in detail~\cite{BGGLP25}.
Furthermore, it was demonstrated in~\cite{H25} that at least in some special cases the eigenvalues of the Robin Laplacian on an exterior Euclidean domain arise in the asymptotic behaviour of the eigenvalues of the Robin Laplacian on a domain in a sphere, which expands to the whole sphere. 

The main motivation of the present paper is to obtain a counterpart 
of the aforementioned Euclidean results~\cite{KL18, KL20}
for exterior domains in the hyperbolic plane. 
In the hyperbolic setting, there are new challenges such as the essential spectrum starting not from zero and the non-criticality of the Neumann Laplacian in the exterior of bounded convex domains. The spectral optimality of geodesic disks is thus by far not evident.

\subsection{Definition of the operator and the characterisation of its lowest spectral point}
Let $\dH^2$ denote the hyperbolic plane, which can be viewed in the hyperbolic model as the upper sheet of the hyperboloid of two sheets equipped with the pullback of the Minkowski metric in~$\RR^3$. For further details on hyperbolic planes see Subsection~\ref{ssec:hyperbolic}. We denote by $\mu$ the Riemannian measure on $\dH^2$ induced by its (Riemannian) metric.

Let $\Omg\subset\dH^2$ be a bounded simply-connected Lipschitz domain with boundary $\p\Omg$. In the following, $\s$ stands for the natural one-dimensional Hausdorff measure on $\p\Omg$. We denote by $\Omg^\ext := \dH^2\sm\ov\Omg$ the complement of $\Omg$ in the hyperbolic plane. 
The Lebesgue space $L^2(\Omg^\ext,\dd\mu)$ 
and the Sobolev space $W^{1,2}(\Omg^\ext,\dd\mu)$ are defined in the standard way. For any $u\in W^{1,2}(\Omg^\ext,\dd\mu)$, its trace $u|_{\p\Omg^\ext}$ is well defined as a function in $L^2(\p\Omg,\dd\s)$. In the following we will omit the indication of the measure $\dd\mu$ and we will write $L^2(\Omg^\ext)$ and $W^{1,2}(\Omg^\ext)$ instead of $L^2(\Omg^\ext,\dd\mu)$ and $W^{1,2}(\Omg^\ext,\dd\mu)$, respectively. We will also omit the indication of measure $\dd\s$ and write $L^2(\p\Omg)$ instead of $L^2(\p\Omg,\dd\s)$.

Let the boundary parameter $\aa\in\dR$ be fixed.
We show in Proposition~\ref{prop:form} that the quadratic form
\begin{equation}\label{key}
\begin{aligned}
	Q_\aa^{\ext}[u] := \int_{\Omg^\ext}|\nabla u|^2\dd\mu + \aa\int_{\p\Omg^\ext}|u|^2\dd\s,\qquad \dom Q_\aa^{\ext} := W^{1,2}(\Omg^\ext),
\end{aligned}
\end{equation}
is closed, densely defined, symmetric and lower-semibounded in $L^2(\Omg^\ext)$. Thus, it induces a unique self-adjoint operator in the Hilbert space $L^2(\Omg^\ext)$, which we denote by  $-\Delta_\aa^{\Omg^\ext}$. We also denote by 
$\lm_1^\aa(\Omg^\ext) := \inf\s(-\Delta_\aa^{\Omg^\ext})$ the lowest spectral point of this operator. By the min-max principle it admits the variational characterisation
\[
	\lm_1^\aa(\Omg^\ext) = \inf_{ u\in W^{1,2}(\Omg^\ext)\sm\{0\}} \frac{\displaystyle
		\int_{\Omg^\ext}|\nabla u|^2\dd\mu + \aa\int_{\p\Omg^\ext}|u|^2\dd\s}{\displaystyle 
		\int_{\Omg^\ext}|u|^2\dd\mu}.
\]
This lowest spectral point can be either a discrete eigenvalue or the bottom of the essential spectrum depending on the value of the boundary parameter $\aa$. 
\subsection{Main results}
Our first main result provides a characterisation of the essential spectrum for the operator $-\Delta^{\Omg^\ext}_\aa$. Not surprisingly, the essential spectrum of this operator is the same as the essential spectrum of the Laplacian on the whole hyperbolic plane. 
\begin{named}{A}\label{thmA}
	$\sess(-\Delta_\aa^{\Omg^\ext}) = [\frac14,\infty)$ for all $\aa \in\dR$.
\end{named}
Theorem~\ref{thmA} will be stated as Proposition~\ref{prop:ess} below and proved by showing two inclusions of the sets, where one inclusion is demonstrated by means of a proper Weyl sequence, while the other inclusion relies on the Neumann bracketing.

We also demonstrate in Lemma~\ref{e1} that $\lm_1^\aa(\Omg^\ext)\arr -\infty$ as $\aa\arr-\infty$ and that $\dR\ni\aa\mapsto \lm_1^\aa(\Omg^\ext)$ is a non-decreasing function. Thus, it is natural to introduce and study the quantity
\[
	\aa_\star(\Omg^\ext) := \sup\left\{\aa\in\dR\colon \ \lm_1^\aa(\Omg^\ext) < \frac14\right\}.
\]
Clearly, $\lm_1^\aa(\Omg^\ext)$ is thus a discrete eigenvalue of $-\Delta_\aa^{\Omg^\ext}$ if, and only if, $\aa < \aa_\star(\Omg^\ext)$. Our second main result concerns this critical constant.
\begin{named}{B}\label{thmB}
	The following hold.
	\begin{myenum}
	\item $\aa_\star(\Omg^\ext) \le 0$
	for any bounded geodesically convex $C^2$-smooth domain 
	$\Omg\subset\dH^2$. If the domain $\Omg$ is, in addition, strictly convex, then $\aa_\star(\Omg^\ext) < 0$. 
	\item $\aa_\star(B_R^\ext) \le \frac12(e^{-R} - \coth{R})$
	for the geodesic disk $B_R\subset\dH^2$ of radius $R > 0$.
	\end{myenum}
\end{named}
Item~(i) of Theorem~\ref{thmB} will be stated in Theorem~\ref{alphaubg} and proved using the method of parallel coordinates with the help of certain one-dimensional Poincar\'e-type inequalities. Item~(ii) in Theorem~\ref{thmB} will be shown in Lemma~\ref{alphaub} by means of separation of variables and of subsequent reduction of existence of the discrete eigenvalue to solvability of a scalar equation involving special functions. We also conjecture that for some non-convex bounded domains in the hyperbolic plane, the critical boundary parameter $\aa_\star(\Omg^\ext)$ can be positive. A geometric construction to support this conjecture is outlined in Remark~\ref{rem:nonconvex}.

The central result of the present manuscript is a comparison between the lowest eigenvalue of the Robin Laplacian on the exterior of a bounded geodesically convex  domain in the hyperbolic plane and the same eigenvalue for the exterior of the geodesic disk. In the following, we use for a domain $\Omg\subset\dH^2$ the notations $|\Omg|$ and $\cH^1(\p\Omg)$ for the area and the perimeter, respectively.
As the notation suggests, 
the latter is introduced through the Hausdorff
measure of the topological boundary~$\partial\Omega$,
but it also coincides with the notion of perimeter
in the sense of De Giorgi 
(\ie, the Hausdorff measure of the reduced boundary)
under the present regularity hypotheses.
\begin{named}{C}\label{thmC}
	Let $\Omg\subset\dH^2$ be a bounded geodesically convex domain and let $B\subset\dH^2$ be a geodesic disk such that
	\begin{equation}\label{eq:geometric_constraint}
		\frac{|\Omg|+2\pi}{\cH^1(\p\Omg)} 
		\le \frac{|B|+2\pi}{\cH^1(\p B)}.
	\end{equation}
	Then, for any $\aa\in\dR$, it holds
	\[
		\lm_1^\aa(\Omg^\ext) \le \lm_1^\aa(B^\ext).
	\]
	In particular, the reverse inequality for the critical boundary parameters $\aa_\star(\Omg^\ext)\ge \aa_\star(B^\ext)$ holds.
\end{named}
We remark that by the Gauss-Bonnet theorem the quantity $|\Omg|+2\pi$ is equal for smooth $\Omg$ to $\int_{\p\Omg}\kp$, 
where $\kp \geq 0$ is the geodesic curvature of $\p\Omg$ with a suitable sign convention. Thus, the quantity
$(|\Omg|+2\pi)/\cH^1(\p\Omg)$ can be interpreted as the averaged geodesic curvature of $\p\Omg$.
The above optimisation result will be restated in Theorem~\ref{thm:main} below and proved using the trial functions, which depend only on the distance to the boundary of the domain. 
The condition~\eqref{eq:geometric_constraint} appears naturally in the analysis through the formula for the length of an outer parallel curve for a bounded convex domain in $\dH^2$. An additional ingredient in the proof is the monotonicity of a certain quotient expressed through the ground-state eigenfunction for the Robin Laplacian on the exterior of the geodesic disk in the hyperbolic plane. 
The monotonicity of the same quotient implies also that the function $R\mapsto \lm_1^\aa(B_R^\ext)$ (here $R$ is the radius of the disk) is non-decreasing. A more precise variant of this property is stated and proved in Proposition~\ref{mpeig} below. 

Theorem~\ref{thmC} implies the isoperimetric inequalities for fixed perimeter or fixed area constraints. Indeed, one can show that the geodesic disk having the same area or the same perimeter as $\Omg$ satisfies the condition~\eqref{eq:geometric_constraint}.
Thus, we can also formulate the following corollary.
\begin{CorA}
	Let $\Omg\subset\dH^2$ be a bounded geodesically convex domain and let $B\subset\dH^2$ be the geodesic disk such that either $|\Omg| = |B|$ or $\cH^1(\p\Omg) = \cH^1(\p B)$.
	Then, for any $\aa\in\dR$, it holds
	\[
	\lm_1^\aa(\Omg^\ext) \le \lm_1^\aa(B^\ext).
	\]
\end{CorA} 

\subsection{Structure of the paper}
In Section~\ref{sec:prelim} we collect some necessary notation on Riemannian manifolds and basic properties of the hyperbolic plane. Section~\ref{sec:specproblem} is devoted to general spectral properties of the Robin Laplacian in the exterior of a bounded Lipschitz domain in the hyperbolic plane. In that section, we characterise the essential spectrum of $-\Delta_\aa^{\Omg^\ext}$ and address the existence of discrete eigenvalues. In Section~\ref{sec:disk} we analyse in detail the Robin Laplacian in the exterior of the geodesic disk in the hyperbolic plane. The central result of our paper (Theorem~\ref{thmC} in the introduction) and its main corollary are stated and proved in Section~\ref{sec:main}. The paper is complemented by three appendices. In Appendix~\ref{app:convexity}, we recall the notions of convex sets in Riemannian manifolds. In Appendix~\ref{app:steiner}, we recall the Steiner formula. Finally, Appendix~\ref{app:Sobolev} is devoted to some fine properties of functions in the Sobolev space $W^{1,2}$ on exterior domains in the hyperbolic plane.

\section{Preliminaries}
\label{sec:prelim}
\subsection{Basic notions on manifolds}
Let $(M, g)$ be a smooth Riemannian manifold of dimension $n$. We define the distance function $d_M$ associated with $g$ by
\[
d_M(p, q) := \inf \left\{ \int_0^1 \sqrt{g(\gamma'(t), \gamma'(t))} \dd t \,\colon\, \gamma \in C^\infty([0,1]; M),\ \gamma(0)=p,\ \gamma(1)=q \right\},\qquad p,q\in M.
\]
We denote by $\mu$ the Riemannian measure induced by $g$, and we write the volume of a measurable set $E \subset M$ as
\[
|E| := \int_E \dd\mu.
\]
\begin{definition}
	Let $K$ be a compact subset of a Riemannian manifold $M$. For every $t \ge 0$, the \emph{outer parallel body} of $K$ at distance $t$ is defined as
	\[
	K_t = \left\{ p \in M \colon d_M(p, K) \le t \right\}.
	\]
\end{definition}
The topological boundary $\partial K_t$ is called the \textit{outer parallel set} of $K$ at distance $t$.
\begin{definition}[Hausdorff distance]
	Let $K_1, K_2 \subset M$ be two compact subsets of a Riemannian manifold $M$. The \emph{Hausdorff distance} between $K_1$ and $K_2$ is defined as
	\[
	d^H(K_1, K_2) = \inf\{ t \ge 0 \,\colon\, K_1 \subset (K_2)_t,\,K_2 \subset (K_1)_t\}.
	\]
\end{definition}

\subsection{Properties of the hyperbolic plane}
%
\label{ssec:hyperbolic}
The theory of hyperbolic planes and of hyperbolic spaces in higher dimensions can be found, \eg, in~\cite{BP92}. In this subsection, we outline basic properties of hyperbolic planes relying on the monograph~\cite{BP92} and also on~\cite[Section 2]{DSZ24}.  

The two-dimensional hyperbolic plane $\dH^2$ is the unique simply-connected, two-dimensional complete Riemannian manifold with constant negative Gaussian curvature equal to $-1$. The uniqueness of the hyperbolic plane, which follows from the Killing–Hopf theorem~\cite[Theorem~12.4]{L18}, implies that any two two-dimensional Riemannian manifolds with these properties are isometric. 

There are several ways to construct hyperbolic planes. We consider here the hyperboloid model.
Given $x,y\in \mathbb{R}^3$, with $x=\left(x_0,x_1,x_2\right)$ and $y=\left(y_0,y_1,y_2\right)$, the Lorentz inner product on $\mathbb{R}^3$ is defined by
\begin{equation}
\langle x,y\rangle=-x_0y_0+x_1y_1+x_2y_2.\nonumber
\end{equation}
The branch of the hyperboloid given by
\begin{equation}
\left\{x\in \mathbb{R}^3\,\colon\,\langle x,x\rangle=-1,x_0>0\right\},\nonumber
\end{equation}
and equipped with the metric tensor
\begin{equation}
\dd l^2=-\dd x_0^2+\dd x_1^2+ \dd x_2^2
\end{equation}
induced by the Lorentz inner product on $\mathbb{R}^3$, is a model of $\mathbb{H}^2$. According to~\cite[Section 4.2]{GV93} the hyperbolic distance between two points $x,y\in\mathbb{H}^2$ is
\begin{align}\label{dh2}
    d_{\mathbb{H}^2}(x,y)={\rm arcosh}\,(-\langle x,y\rangle).
\end{align}
In this model, the geodesic starting at $x\in \mathbb{H}^2$ with unit velocity $y\in T_x\mathbb{H}^2$ is given by 
$$
\dR\ni t\arr \cosh(t)x+\sinh(t)y\in \mathbb{H}^2,
$$
that is the intersection of $\mathbb{H}^2$ with the plane in $\mathbb{R}^3$ spanned by $x$ and $y$; \cf~\cite[Proposition~A.5.1]{BP92}.

Following the lines of~~\cite[Section 2]{DSZ24}
we fix the point $O = (1, 0, 0)$ as the origin of $\mathbb{H}^2$ and adopt the geodesic polar ``coordinates''
$(r,\theta)\in [0,\infty)\times \mathbb{S}^1$ centred at $O$. Any point $(x_0, x') \in \mathbb{H}^2$, with $x' = (x_1,x_2)$, can be represented as
\[
(x_0,x') = \left(\cosh r,\theta\sinh r\right)\,,
\]
where $r\geq 0$ denotes the hyperbolic distance from $O$ to $(x_0,x')$.
Since
\begin{equation}
\dd x_0=\sinh r\,\text{d}r,\qquad \dd x'=\theta\cosh r \dd r+\sinh r\dd \theta,\nonumber
\end{equation}
then the metric takes the form
\begin{equation}
\text{d}l^2=\text{d}r^2+\sinh^2r\,\text{d}\theta^2.\nonumber
\end{equation}
Finally, the Laplace-Beltrami operator $\Delta$ on $\mathbb{H}^2$ in these geodesic polar coordinates is given by
\begin{equation}
\Delta = \partial_r^2+\frac{\cosh r}{\sinh r}\partial_r+\frac{\p_\tt^2}{\sinh^2 r}\,.\nonumber
\end{equation}

Next, we address the properties of convex sets in the hyperbolic plane. The convexity of a set in a Riemmannian manifold is more subtle than in the Euclidean space. There are three related notions of convexity: \emph{weak}, \emph{strong} and \emph{local} convexity, which are defined in Appendix~\ref{app:convexity}. In the hyperbolic plane, the notions of a weakly convex and a strongly convex set coincide, since for every two points there exists a unique minimising geodesic connecting them. Therefore, in the following sections, we will refer simply to geodesically convex sets.
\begin{Proposition}
	\label{prop: localtostrong}
	Let $C\subset\mathbb{H}^2$ be a closed, connected, locally convex set. Then $C$ is strongly convex.
\end{Proposition}
\begin{proof}
	The local convexity of $C$, combined with its connectedness, guarantees that the interior $\mathring{C}$ is connected (see \cite[Lemma 1.5]{CG72}). This allows us to use Theorem \ref{teor: supporting} for $\mathring{C}$, ensuring that every boundary point $p \in \partial \mathring{C}$ possesses a supporting element. Since in hyperbolic plane the cut locus $\cut(p)$ is empty for all points $p$, we have
	\[
	\mathring{C}\setminus \cut(p)=\mathring{C},
	\]
	which is connected. Therefore, by applying Proposition \ref{prop: weakchar}, it follows that $\mathring{C}$ is weakly convex, and consequently strongly convex, where we implicitly used that there is a unique minimal geodesic between any two points in $\dH^2$. 
	Lastly, because $C$ is closed and locally convex, it coincides with the closure of its interior (see \cite[Theorem 1.6]{CG72}), and so $C$ inherits the strong convexity property from $\mathring{C}$.
\end{proof}
Convex subsets of $\mathbb{H}^2$ can be approximated, with respect to the Hausdorff distance, by convex sets of class $C^\infty$. 
\begin{Proposition}
\label{prop: approx}
    Let $C\subset\mathbb{H}^2$ be a closed, bounded, strongly convex set such that $\mathring{C}\ne \emptyset$. Then there exists a sequence of closed, bounded and strongly convex sets $\{C_k\}_{k\in\dN}$ with $C^{\infty}$ boundaries such that
    \[
        \lim_{k\to\infty} \big(d^H(C_k,C)+d^H(\partial C_k,\partial C)\big)=0.
    \]
\end{Proposition}
\begin{proof}
    Since the Gaussian curvature in $\mathbb{H}^2$ is negative, Theorem \ref{teor: approx} applies and we find an approximating sequence of connected, compact, locally convex sets $\{C_k\}_{k\in\dN}$ with $C^\infty$ boundaries and such that $\mathring{C}_k\ne \emptyset$. Proposition \ref{prop: localtostrong} ensures that $C_k$ is strongly convex for every $k\in\dN$.
\end{proof}

Let $\Omega\subset\mathbb{H}^2$ be a connected, compact, locally convex set. Then by Proposition~\ref{lowboundreach}\,(ii) the reach of the set $\Omega$ is given by $\cR(\Omega) = +\infty$. 
Assuming, in addition that $\Omega$ is $C^2$-smooth and using Steiner's formula in Theorem~\ref{Theorem: steiner}
combined with the Gauss--Bonnet theorem~\cite[\S III.1]{Ch84},
we get that the one-dimensional Hausdorff measure
of the outer parallel set $\p\Omg_t$ is given by
\begin{equation}\label{eq:H1pOmgt}
\begin{aligned}
		\cH^1(\p\Omg_t) &= \cosh(t)\cH^1(\p\Omg) +\sinh(t) \int_{\p\Omg} H_1(p) \dd\cH^1(p)\\
		&=\cosh(t)\cH^1(\p\Omg) + \sinh(t)\big(2\pi + |\Omg|\big),
\end{aligned}	
\end{equation}
where $H_1 \geq 0$ is the geodesic curvature of $\p\Omg$.

We now extend the identity~\eqref{eq:H1pOmgt} from $C^2$-smooth bounded convex sets in $\dH^2$ to general bounded convex sets without any regularity assumption on the boundary. The argument is performed via an approximation of non-smooth sets by smooth ones.
\begin{Theorem}\label{thm:nonsmooth}
Let $\Omega\subset \mathbb{H}^2$ be a bounded, geodesically convex domain. Then \[
\mathcal{H}^{1}\left(\partial\Omega_t\right) =
\cosh(t)\cH^1(\p\Omg) + \sinh(t)\big(2\pi + |\Omg|\big) ,\qquad\text{for all}\,\, t > 0.
\] 
\end{Theorem}
\begin{proof}
	 We will use the quantities $\{\Phi_k(\cdot)\}_k$, which are defined and analysed in Appendix~\ref{app:steiner}.
    It follows from Remark~\ref{per} that $\Phi_1(\Omega) = \cH^1(\p \Omega)$.  By applying Proposition~\ref{prop: approx} to $\overline{\Omega}$, there exists an approximating sequence $\{C_k\}_{k\in\mathbb{N}}$  of closed, connected and strongly convex smooth sets. From the Gauss-Bonnet Theorem, we deduce 
	$$\Phi_0(C_k)=2\pi+\abs{C_k}=2\pi+\Phi_2(C_k),$$
    then it holds    
	$$\Phi_0(\Omega)=2\pi+\abs{\Omega},$$
as a consequence of Theorem \ref{teor: curvconv}. 
    By using Steiner's formula in Theorem~\ref{Theorem: steiner}, we get the claim.
\end{proof}
We also recall the classical isoperimetric inequality in the hyperbolic plane, which holds for sets of finite perimeter in the sense of De Giorgi (see, e.g., \cite[Theorem 5.16]{R23}; see also \cite{BDS} for its sharp quantitative version). Here, we state it in the restricted framework of convex sets, where this notion of perimeter coincides with the $\mathcal{H}^1$-measure of the topological boundary, as shown in Remark \ref{per}.
\begin{Theorem}\label{corisop1}
	Let $\Omega\subset \mathbb{H}^2$ be a geodesically convex set and $B$ be a geodesic disk with the same area as $\Omega$ then
	$$\mathcal{H}^{1}(\p \Omg)\geq \mathcal{H}^1(\p B),$$
	where the equality is attained if and only if $\Omega$ is a geodesic disk.  
\end{Theorem}
\begin{Remark}
	The area and the perimeter of a geodesic disk of radius $r$ in $\dH^2$ are given by
	$$|B_r|=2\pi \big(\cosh r -1\big),\qquad \mathcal{H}^1(\partial B_r)=2\pi\sinh r.$$
\end{Remark}
\begin{Remark}
	If $B_1(x_0)$ is the geodesic disk centred at $x_0$ having the same perimeter as $\Omega$ and $B_2(x_0)$ denotes the geodesic ball centred in the same point and having same area as $\Omega$ then from the isoperimetric inequality we get that $B_2(x_0)\subset B_1(x_0)$. 
\end{Remark}

\section{The spectral problem on an exterior domain in the hyperbolic plane}\label{sec:specproblem}
%
Let $\Omega\subset \mathbb{H}^2$ be an open, bounded and Lipschitz set such that the exterior $\Omg^\ext$ is connected. We aim at considering the following 
spectral problem
\begin{equation}\label{robextprob}\begin{cases}
    -\Delta u=\lambda u &\text{in }\Omg^\ext,\\[5 pt]
    -\dfrac{\partial u}{\partial n}+\alpha u=0 &\text{on }\partial\Omg^\ext,
\end{cases}\end{equation}
where $n$ denotes the outer unit normal to $\Omg$ and $\alpha\in\mathbb{R}$ is a fixed boundary parameter.
In order to introduce this spectral problem rigorously, we consider the sesquilinear form
$$Q_{\alpha}^{\text{ext}}[u]:=\int_{\Omg^\ext}\abs{\nabla u}^2\dd\mu+\alpha\int_{\partial \Omg^\ext}|u|^2\dd\sigma,\qquad\text{dom}(Q_{\alpha}^{\text{ext}}):=W^{1,2}(\Omg^\ext),$$
where the boundary term is understood in the sense of traces $W^{1,2}(\Omg^\ext)\hookrightarrow L^2(\partial\Omg^\ext)$. The fact that the trace mapping from $W^{1,2}(\Omega^\ext)$ into $L^2(\partial\Omega^\ext)$ is well defined and continuous follows immediately by restricting the Sobolev functions to a compact Lipschitz set $K\setminus \Omega$ with some compact $K\supset\Omg$ and applying the standard
trace theorem for the case of bounded domains; see \eg  \cite[Theorem 7.5]{BM13}. It can be shown that the sesquilinear form is symmetric, densely defined, closed, and lower semibounded 
in the Hilbert space $L^2(\Omg^\ext)$.
The facts that the form $Q_\alpha^\ext$ is symmetric and densely defined are clear. Below we focus on proving that this form is closed and lower semibounded.
\begin{Proposition}\label{prop:form}
	The sesquilinear form $Q_\aa^\ext$ is closed and lower semibounded.
\end{Proposition}
\begin{proof}
First, we note that $Q_0^\ext$ ($\aa = 0$) is closed
and non-negative. While non-negativity is obvious, closedness follows from the fact that the norm induced by this form coincides with the standard norm in the Sobolev space $W^{1,2}(\Omg^\ext)$. 
By Lemma~\ref{lem:Ehrling} in Appendix~\ref{app:Sobolev} for any $\eps > 0$ there exists a constant $C(\eps) > 0$ such that the inequality
\[
	\int_{\p\Omg^\ext}|u|^2\dd \s \le \eps\int_{\Omg^\ext}|\nb u|^2\dd \mu + C(\eps)\int_{\Omg^\ext}|u|^2\dd\mu
\]
holds for all $u\in W^{1,2}(\Omg^\ext)$. 
Hence, the above inequality with a sufficiently small $\varepsilon >0$ combined with stability result for the closedness of forms under perturbations~\cite[Chapter~VI, Theorem~3.4]{Ka} yields that the sesquilinear form $Q_\aa^\ext$ is closed and lower semibounded for any $\aa\in\dR$.
\end{proof}

The first representation theorem \cite[Chapter VI, Theorem 2.1]{Ka} ensures that the form $Q_\aa^\ext$ defines a unique self-adjoint operator $-\Delta^{\Omg^\ext}_{\alpha}$ in
the Hilbert space $L^2(\Omg^\ext)$ called Robin Laplacian and acting as
\begin{align*}
    -\Delta^{\Omg^\ext}_{\alpha}u&=-\Delta u,\\
     \dom\left( -\Delta^{\Omg^\ext}_{\alpha}\right)&=\left\{u\in W^{1,2}(\Omega^\ext)\,\colon\, \ \Delta u\in L^2(\Omega^\ext), \ -\frac{\p u}{\p n}\Big|_{\partial\Omega^\ext}+\alpha u|_{\partial\Omg^\ext}=0\right\},
\end{align*}
where the Neumann trace $\frac{\p u}{\p n}\Big|_{\partial\Omega^\ext}$ is understood 
as a distribution in $ W^{-1/2,2}(\p\Omg^\ext)$. 
The formal problem \eqref{robextprob} is then meant as the spectral problem for the self-adjoint operator $-\Delta^{\Omg^\ext}_{\alpha}$ in $L^2(\Omg^\ext)$.
The embedding $W^{1,2}(\Omg^\ext)\hookrightarrow L^2(\Omg^\ext)$ is not compact, as shown in the following remark.
\begin{Remark}
    Let $\varphi \in C_c^\infty(B_1(p))$ be a non-zero cut-off function. 
Choose points $p_k \in \Omega^\ext$ with 
$d_{\dH^2}(p_k,p)\to\infty$ and $d_{\dH^2}(p_k,p_j) > 2$ for $k\neq j$. 
Pick isometries $g_k \in \mathrm{Isom}(\mathbb H^2)$ with $g_k(p) = p_k$ and set $\varphi_k := \varphi \circ g_k^{-1}.$ Then $\mathrm{supp}\,\varphi_k \subset B_1(p_k)$ for all $k\in\dN$, the supports of $\{\varphi_k\}$ are mutually disjoint, 
$\|\varphi_k\|_{W^{1,2}(\Omg^\ext)} 
= \|\varphi\|_{W^{1,2}(\Omg^\ext)}$
and 
$\|\varphi_k - \varphi_j\|_{L^2(\Omg^\ext)} 
= \sqrt{2}\,\|\varphi\|_{L^2(\Omg^\ext)}$ 
for $k\neq j$. Thus $\varphi_k$ is bounded in $W^{1,2}(\Omg^\ext)$
but has no convergent subsequence in $L^2(\Omg^\ext)$.
\end{Remark}

Therefore the resolvent of the Robin Laplacian $-\Delta_\aa^{\Omg^\ext}$ is not a compact operator.
We will show that the essential spectrum of the Robin Laplacian $-\Delta^{\Omg^\ext}_{\alpha}$ is non-empty. It is also useful for the following considerations to recall that a variational characterisation of the lowest point of the spectrum \cite[Theorem XIII.1]{RS} holds 
\begin{align}\label{lambda1}
\lambda_1^\alpha(\Omg^\ext):=\inf\sigma\left(-\Delta^{\Omg^\ext}_{\alpha}\right)=\inf_{\substack{u\in W^{1,2}(\Omg^\ext)\\ u\not = 0}}\frac{Q_{\alpha}^{\text{ext}}[u]}{\|u\|_{L^2(\Omg^\ext)}^2}.
\end{align}
Next, we prove a weighted Poincaré-type inequality with a boundary term. This inequality will be used later in the proof of the characterisation of the essential spectrum for the Robin Laplacian $-\Delta^{\Omg^\ext}_\aa$.
\begin{Lemma}\label{auxlem3}
     Let $b>0$ and $\alpha\geq 2^{-1}(b^{-1}-\coth(b))$. Then
   \begin{align*}
     \inf_{\substack{\psi \in C_c^1([b,\infty)) \\ \psi \not= 0}}\frac{\displaystyle\int_b^{\infty} |\psi'(t)|^2 \sinh(t) \dd t+\alpha\sinh(b)|\psi(b)|^2}{\displaystyle\int_b^{\infty} |\psi(t)|^2 \sinh(t) \dd t}\geq \frac14.
   \end{align*}
\end{Lemma}
\begin{proof}
First, we derive an auxiliary inequality.
Note that 
$$\forall f\in C_c^{1}([b,\infty)) \,, \quad
 \int_b^{\infty} |f'(t)|^2\, t\dd t \geq 0.$$
By making the substitution 
$g(t) = \sqrt{t} \, f(t)$, 
we get
\begin{align}\label{recalling}
   \int_b^{\infty} |f'(t)|^2\, t \dd t&=\int_b^{\infty} \left|\frac{g'}{\sqrt{t}} - \frac{1}{2}\frac{g}{t^{3/2}}\right|^2 t \dd t
   \nonumber \\ 
   &= \int_b^{\infty} \left|g' - \frac{1}{2}\frac{g}{t}\right|^2 \dd t 
   \nonumber \\
   &= \int_b^{\infty} |g'|^2 \dd t - \frac{1}{2}\int_b^{\infty} (|g|^2)' \frac{1}{t} \dd t + \frac{1}{4}\int_b^{\infty} \frac{|g|^2}{t^2} \dd t
   \nonumber \\
   &= \int_b^{\infty} |g'|^2 \dd t + \frac{1}{2b}|g(b)|^2 - \frac{1}{4}\int_b^{\infty} \frac{|g|^2}{t^2} \dd t \geq 0.
\end{align}
For any $\psi\in C^1_c([b,\infty))$  not identically zero, we set 
$\varphi(t)=\sqrt{\sinh(t)} \, \psi(t)$, 
obtaining
\begin{align}
    \int_b^{\infty} |\psi'(t)|^2 \sinh(t) \dd t&= \int_b^{\infty} \left|\frac{\varphi'}{\sqrt{\sinh(t)}} - \frac{1}{2}\frac{\varphi \cosh(t)}{\sinh^{3/2}(t)}\right|^2\sinh(t) \dd t\nonumber\\&= \int_b^{\infty} \left|\varphi' - \frac{1}{2}\varphi \coth(t)\right|^2 \dd t\nonumber\\&=\int_b^{\infty} |\varphi'|^2 \dd t - \frac{1}{2}\int_b^{\infty} (|\varphi|^2)' \coth(t) \dd t + \frac{1}{4}\int_b^{\infty} |\varphi|^2 \coth^2 (t) \dd t \nonumber\\&=\int_{b}^{\infty}|\varphi'|^{2}\dd t+\frac{\coth (b)|\varphi(b)|^{2}}{2}+\frac{1}{4}\int_{b}^{\infty}\left(\coth^{2}(t)-\frac{2}{\sinh^{2}(t)}\right)|\varphi|^{2}\dd t.\label{fineq}
\end{align}
Using  $$\coth^{2}(t)-\frac{2}{\sinh^{2}(t)}=1-\frac{1}{\sinh^{2}(t)}\geq 1-\frac{1}{t^2},$$
and recalling ~\eqref{recalling}, we add and subtract the term $\frac{1}{2b}|\varphi(b)|^2$ in \eqref{fineq} to obtain
\begin{align*} \int_b^{\infty} |\psi'(t)|^2 \sinh(t) \dd t+\alpha\sinh(b)|\psi(b)|^2&\geq \frac14\int_b^{\infty}|\varphi|^2\dd t+\left(\frac{\coth(b)}{2}-\frac{1}{2b}+\alpha\right)|\varphi(b)|^2\\&\geq \frac14\int_b^{\infty}|\varphi|^2\dd t \\&= \frac14\int_b^{\infty}|\psi|^2\sinh(t)\dd t.\qedhere
\end{align*}
\end{proof} 
In the next proposition, we characterise the essential spectrum of the Robin Laplacian in the exterior of a bounded Lipschitz set in the hyperbolic plane. As expected, this essential spectrum coincides with the essential spectrum of the Laplacian in the whole hyperbolic plane. 
\begin{Proposition}\label{prop:ess}
The essential spectrum of the Robin Laplacian $-\Delta^{\Omg^\ext}_{\alpha}$ is given by
  \begin{align*}
      \sess\left(-\Delta^{\Omega^\ext}_{\alpha}\right)=\left[\frac{1}{4}, \infty\right).
  \end{align*}  
\end{Proposition}
\begin{proof}
\noindent \emph{Step 1: the inclusion $\sess(-\Delta_\aa^{\Omg^\ext}) \supseteq \left[\frac14,\infty\right)$.}	
We apply Weyl's criterion to prove the inclusion $\supseteq$. For every $ \lambda \in \left[\frac{1}{4}, \infty\right)$, we will construct a sequence $ f_n \in \dom (-\Delta_\aa^{\Omg^\ext})$ such that
\begin{enumerate}
    \item $ \|f_n\|_{L^2(\Omega^\ext)} = 1 $,
    \item $ \|(-\Delta^{\Omega^\ext}_{\alpha} - \lambda)f_n\|_{L^2(\Omg^\ext)} \to 0 $ as $ n \to \infty $.
    \item $f_n\rightharpoonup 0$ in $L^2(\Omg^\ext)$.
\end{enumerate}
This construction follows the approach outlined in \cite[Proof of Proposition 7.2]{BOR}. Given $\lambda \geq \frac{1}{4}$, we set
\[
s = \frac{1}{2} + i\rho, \quad \text{where } \rho = \sqrt{\lambda - \frac{1}{4}},
\]
then $\lambda = s(1 - s)$. We define for  $2\leq a<b$ a family of compactly supported functions on $\dH^2$ in geodesic polar coordinates by
\[
f_{a,b}(r) := e^{-s r} \psi_{a,b}(r),
\]
where $ \psi_{a,b} \in C^\infty_c((0,\infty))$ is a radial cut-off function such that 
\[
\psi_{a,b}(r) =
\begin{cases}
1, & r \in [a, b], \\
0, & r \notin [a-1, b+1],
\end{cases}
\quad \text{with } 0\le \psi_{a,b}\le 1,\, \|\psi_{a,b}'\|_\infty,\, \|\psi_{a,b}''\|_\infty \leq C_1,
\]
for some constant $C_1 > 0$ independent of $ a, b $ and taken, without loss of generality, greater than~$1$.
A straightforward calculations gives 
\begin{align*}
f_{a,b}'(r) &= -s e^{-s r} \psi_{a,b}(r) + e^{-s r} \psi_{a,b}'(r), \\
f_{a,b}''(r) &= s^2 e^{-s r} \psi_{a,b}(r) - 2s e^{-s r} \psi_{a,b}'(r) + e^{-s r} \psi_{a,b}''(r).
\end{align*}
Since $ f_{a,b} $ is radial, we compute the Laplacian of $f_{a,b}$ in geodesic polar coordinates
\begin{align*}
-\Delta f_{a,b}(r)
&=- f_{a,b}''(r) - \coth r \cdot f_{a,b}'(r) \\
&= -e^{-s r} \left[ \left( s^2 - s \coth r \right) \psi_{a,b}(r) + \left( \coth r - 2s \right) \psi_{a,b}'(r) + \psi_{a,b}''(r) \right].
\end{align*}
Therefore
\[
-\Delta f_{a,b} - \lambda f_{a,b}
= -e^{-s r} \left[ s\left( 1-\coth r  \right) \psi_{a,b} + (\coth r - 2s)\psi_{a,b}' + \psi_{a,b}'' \right].
\]
We estimate the norm of $ f_{a,b} $ in $L^2(\dH^2)$ as
\begin{equation}\label{eq:L2ab}
\|f_{a,b}\|^2_{L^2(\mathbb{H}^2)} = 2\pi \int_{0}^\infty |f_{a,b}(r)|^2 \sinh r \dd r \geq  2\pi \int_{a}^{b} e^{- r} \sinh r \dd r\geq \frac{\pi}{2}(b-a).
\end{equation}
\noindent We have
\begin{align*}
    \|(-\Delta- \lambda) f_{a,b}\|_{L^2(\mathbb{H}^2)}^2&=2\pi \int_{a-1}^{b+1} \left| (-\Delta- \lambda) f_{a,b}(r) \right|^2 \sinh r  \dd r . 
\end{align*}
Next, we split the integral in three pieces. Observe that
that for any $c\ge 1$
\begin{equation}\label{eq:bndx}
\begin{aligned}
   &\int_{c}^{c+1} \left| (-\Delta- \lambda) f_{a,b}(r) \right|^2 \sinh r \dd r\\
   &\qquad \le
   3C_1^2\int_c^{c+1} e^{-r}\sinh(r)\Big(|s|^2(1-\coth r)^2 + |\coth r - 2s|^2 + 1\Big)\dd r
   \\
   &\qquad \le \frac32C_1^2\Big(|s|^2(\coth(1) -1)^2 + 2\coth^2(1)+4|s|^2+1\Big).
\end{aligned}
\end{equation}
Moreover
\begin{equation}\label{eq:L2Lapab2}
\begin{aligned}
 	&\int_{a}^{b} \left| (-\Delta -\lambda) f_{a,b}(r) \right|^2 \sinh r \dd r\\
 	&\qquad\leq |s|^2\int_{2}^{\infty}(\coth(r)-1)^2e^{-r}\sinh(r)\dd r\\
 	&\qquad\leq
 |s|^2\int_{2}^{\infty}\frac{2e^{-4r}}{1-e^{-2r}}\dd r\leq \frac{2|s|^2}{1-e^{-4}}\int_0^\infty e^{-4r}\dd r = \frac{|s|^2}{2(1-e^{-4})} 
 \,.
\end{aligned}
\end{equation}
Hence, employing~\eqref{eq:bndx} with $c=a-1$ and $c=b$
and using the last bound in~\eqref{eq:L2Lapab2}, we get that the norm of the quantity $(-\Delta-\lm) f_{a,b}$ in $L^2(\dH^2)$ is bounded by a constant which depends only on $\lambda$, but not on $a$ and $b$.
Now, we define
\[
f_n := \frac{f_{n^2, n^2 + n}}{\|f_{n^2, n^2 + n}\|_{L^2(\mathbb{H}^2)}}.
\]
In this way $ \|f_n\|_{L^2(\mathbb{H}^2)} = 1 $, the supports are disjoint and, at the same time, the support of~$f_n$ leaves any bounded set for all sufficiently large $n$.
Hence, we conclude that $f_n\rightharpoonup 0$ and moreover it follows from the bounds~\eqref{eq:L2ab} and~\eqref{eq:bndx},~\eqref{eq:L2Lapab2} that 
\[
\|(-\Delta - \lambda) f_n\|_{L^2(\dH^2)}\arr 0,\qquad n\arr\infty.
\]
Moreover, for all sufficiently large $n$ we have $\supp f_n\subset\Omg^\ext$. Therefore, we infer that  $f_n\in\dom(-\Delta_\aa^{\Omg^\ext})$ and, in conclusion, $f_n$ is a Weyl sequence for the operator $-\Delta_\aa^{\Omg^\ext}$ corresponding to the point $\lambda$. Thus, $\lambda \in \sess\left(-\Delta^{\Omega^\ext}_{\alpha}\right)$. Since $ \lambda \in [\frac14, \infty)$ was arbitrary, we obtain
\[
\left[ \frac{1}{4}, \infty \right) \subseteq \sess\left(-\Delta^{\Omega^\ext}_{\alpha}\right).
\]
\medskip
\noindent \emph{Step 2: the inclusion $\sess(-\Delta_\aa^{\Omg^\ext}) \subseteq \left[\frac14,\infty\right)$.}
In order to show the opposite inclusion, 
we use a Neumann bracketing argument. Let us introduce the notation $|x| := d_{\dH^2}(x,O)$ for $x\in\dH^2$. Let $H_n$ be the operator that acts as $-\Delta^{\Omg^\ext}_{\alpha}$ but satisfies an extra Neumann condition on the circle $\cC_n := \{ x \in \mathbb{H}^2 \colon |x| = n \}$ of the radius $n \in\dN$. More specifically, $H_n$ is the self-adjoint operator in $L^2(\Omg^\ext)$ associated with the form
\[
h_n[u] := \|\nabla u\|^2_{L^2(\Omega^\ext)} + \alpha \|u\|^2_{L^2(\partial\Omega^\ext)}, \quad \dom(h_n) := W^{1,2}(\Omega^\ext \setminus \cC_n).
\]
Because of the domain inclusion $\dom(h_n) \supset \dom(Q^{\Omg^\ext}_{\alpha})$, we have $-\Delta^{\Omg^\ext}_{\alpha} \ge H_n$ in the sense of ordering of the corresponding sesquilinear forms and, by the min-max principle,
\[
\inf \sess(-\Delta^{\Omg^\ext}_{\alpha}) \ge \inf \sess(H_n) \quad \text{for all } n\in\dN.
\]
Assuming that $n\in\dN$ is sufficiently large so that $\Omega$ is contained in the disk $B_n := \{ x \in \mathbb{H}^2 \colon |x| < n \}$, $H_n$ decouples into an orthogonal sum of two operators, $H_n = H^{(1)}_n \oplus H^{(2)}_n$ with respect to the decomposition $L^2(\Omega^\ext) = L^2(\Omg^\ext \cap B_n) \oplus L^2(\Omega^\ext \setminus \ov{B_n)}$. Here $H^{(1)}_n$ and $H^{(2)}_n$ are respectively the operators in $L^2(\Omega^\ext \cap B_n)$ and $L^2(\Omega^\ext \setminus \ov{B_n})$ associated with the forms
\[
\begin{aligned}
h_n^{(1)}[u] &:= \|\nabla u\|^2_{L^2(\Omega^\ext \cap B_n)} + \alpha \|u\|^2_{L^2(\p\Omg^\ext)}, \qquad &\dom(h_n^{(1)})& := W^{1,2}(\Omega^\ext \cap B_n),\\
h_n^{(2)}[u] &:= \|\nabla u\|^2_{L^2(\Omega^\ext \setminus \overline{B_n})}, \qquad &\dom(h_n^{(2)}) &:= W^{1,2}(\Omega^\ext \setminus \overline{B_n}).
\end{aligned}
\]
Since $\Omg^\ext \cap B_n$ is a Lipschitz bounded open set, the embedding of $W^{1,2}(\Omg^\ext \cap B_n)$ into $L^2(\Omg^\ext \cap B_n)$ is compact and the spectrum of $H_n^{(1)}$ is purely discrete. Consequently,
\begin{align*}
\inf \sess(-\Delta^{\Omg^\ext}_{\alpha}) 
&\ge \inf \sess(H_n^{(2)}) 
\ge \inf \sigma(H_n^{(2)}) \\
&= \inf_{\substack{u \in W^{1,2}(B_n^\ext) \\ u \ne 0}} 
\frac{
    \displaystyle\int_{B_n^\ext} |\nabla u(r,\theta)|^2 \dd\mu 
}{
     \displaystyle\int_{B_n^\ext} |u
     (r,\theta)|^2 \dd\mu
}\\
&\geq \inf_{\substack{u \in W^{1,2}(B_n^\ext) \\ u \ne 0}} 
\frac{
    \displaystyle\int_0^{2\pi}\int_n^{\infty}\left|\frac{\partial}{\partial r}u(r,\theta)\right|^2\sinh(r) \dd r \dd\theta
}{
     \displaystyle\int_0^{2\pi}\int_n^{\infty}|u(r,\theta)|^2\sinh(r) \dd r\dd \theta}\geq \frac14,
\end{align*}
where we used density of $C^1_c(\overline{B_n^\ext})$ in $W^{1,2}(B_n^\ext)$ (see Lemma~\ref{lem:density}) and where the last inequality follows from Lemma~\ref{auxlem3} applied with $\alpha=0$. Indeed, for a given $\theta\in[0,2\pi)$, we set $\varphi(r):=u(r,\theta)\in C^\infty_c(\overline{B_n^{\text{ext}}})$, then we have
$$\int_n^{\infty}|\varphi'(r)|^2\sinh(r)\, \dd r\geq\frac14\int_n^{\infty}|\varphi(r)|^2\sinh(r)\dd r.$$
Thus, the proof is complete.
\end{proof}

Since the essential spectrum is non-empty, it is not necessary that $\lambda_1^\alpha(\Omega^\ext)$ represents a discrete eigenvalue. Using standard arguments, we prove that the quantity $\lambda_1^\alpha(\Omega^\ext)$ is non-decreasing in $\alpha$ and tends to $-\infty$ as $\aa\arr-\infty$.
\begin{Lemma}\label{e1}
    The function $$\alpha\in\mathbb{R}\mapsto\lambda_1^\alpha(\Omega^\ext)\in\mathbb{R}$$ is non-decreasing. Moreover,
    $$\lim_{\alpha\to-\infty}\lambda_1^\alpha(\Omega^\ext)=-\infty.$$
\end{Lemma}
\begin{proof}
    Given $\alpha_1>\alpha_2$ and a minimising sequence $\varphi_n\in W^{1,2}(\Omg^\ext)$ for the operator $-\Delta^{\Omg^\ext}_{\aa_1}$ corresponding to its lowest spectral point $\lambda_1^{\alpha_1}(\Omega^\ext)$, then for every $\eps>0$, there exists $n_0\in\dN$ such that for all $n>n_0$, the following holds
\begin{align*}
    \lambda_1^{\alpha_1}(\Omega^\ext)+\eps&\geq \frac{\displaystyle\int_{\Omg^\ext}\abs{\nabla \varphi_n}^2\dd\mu+\alpha_1\displaystyle\int_{\partial \Omg^\ext}|\varphi_n|^2\dd\sigma}{
\displaystyle \int_{\Omg^\ext}|\varphi_n|^2\dd\mu}\\ &\geq\frac{\displaystyle\int_{\Omg^\ext}\abs{\nabla \varphi_n}^2\dd\mu+\alpha_2\displaystyle\int_{\partial \Omg^\ext}|\varphi_n|^2\dd\sigma}{
\displaystyle \int_{\Omg^\ext}|\varphi_n|^2\dd\mu}\\&\geq \lambda_1^{\alpha_2}(\Omg^\ext).
\end{align*}
To prove the second part of the statement, we fix a test function $\varphi\in W^{1,2}(\Omg^\ext)$ such that $\|\varphi\|_{L^2(\p\Omg^\ext)} \ne 0$. Then
\[
\begin{aligned}
    \lambda_1^{\alpha}(\Omg^\ext)\leq\frac{\displaystyle\int_{\Omg^\ext}\abs{\nabla \varphi}^2\dd\mu+\alpha\displaystyle\int_{\partial \Omg^\ext}|\varphi|^2\dd\sigma}{
\displaystyle \int_{\Omg^\ext}|\varphi|^2\dd\mu}\quad\xrightarrow{\alpha \to -\infty} \quad-\infty.\qedhere
\end{aligned}
\]
\end{proof}
Below we recall the definition from the introduction of the critical boundary parameter.
\begin{definition}\label{astar}
    We introduce the critical value $\alpha_{\star}(\Omega^{\mathrm{ext}})$, defined as
$$
\alpha_{\star}(\Omega^{\mathrm{ext}}) 
:= \sup \{\alpha\in\mathbb{R} \,\colon\, \lambda_1^\alpha(\Omega^{\mathrm{ext}})
\text{ is a discrete eigenvalue}\}.
$$
\end{definition}
\begin{Remark}
    From Lemma~\ref{e1} and Proposition~\ref{prop:ess}, 
    we deduce that if $\alpha$ is sufficiently negative, then the Robin Laplacian operator has discrete eigenvalues. Consequently, Definition \ref{astar} is well posed, since the set over which the supremum is taken is nonempty. Lemma \ref{e1} further ensures that for every $\alpha<\alpha_{\star}(\Omega^{\mathrm{ext}})$, the operator $-\Delta_\aa^{\Omg^\ext}$ admits discrete eigenvalues.
\end{Remark}

In order to prove that $\alpha_{\star}(\Omega^\ext)$ is non-positive when $\Omega$ is geodesically convex, we first establish in the next two lemmas certain weighted Poincaré-type inequalities with a boundary term,
in the spirit of Lemma~\ref{auxlem3}.
\begin{Lemma}\label{auxlem2}
    Let $b>0$ and $\alpha\geq -2^{-1}$. Then
   \begin{align*}
       \inf_{\substack{\psi \in C_c^1([b,\infty)) \\ \psi \not= 0}} \frac{\displaystyle\int_b^{\infty} |\psi'(t)|^2 e^t \dd t+\alpha e^b|\psi(b)|^2}{\displaystyle\int_b^{\infty} |\psi(t)|^2 e^t \dd t}\geq \frac14.
   \end{align*}
\end{Lemma}
\begin{proof}
Let $\psi \in C_c^1([b, \infty))$.  
We set $\varphi(t) := \psi(t) e^{\frac{t}{2}}$. Then we get
\[
\psi'(t) = \varphi'(t) e^{-\frac{t}{2}} - \frac{1}{2} \varphi(t) e^{-\frac{t}{2}}.
\]
Hence, we arrive at
\begin{align*}
    \int_b^{\infty} |\psi'(t)|^2 e^{t} \dd t &= \int_b^{\infty} \left| \varphi'(t) - \frac{1}{2}\varphi(t) \right|^2 \dd t\\&=\int_b^{\infty} |\varphi'(t)|^2 \dd t - \int_b^{\infty}\Re\big( \varphi'(t)\ov{\varphi(t)}\big) \dd t + \frac{1}{4} \int_b^{\infty} |\varphi(t)|^2 \dd t.
\end{align*}
Note that using integration by parts we get
\[
- \int_b^{\infty}\Re\big( \varphi'(t)\ov{\varphi(t)}\big) \dd t = -\frac{1}{2} \int_b^{\infty} \frac{\dd}{\dd t} |\varphi(t)|^2 \dd t = \frac12|\varphi(b)|^2.
\]
Thus, we end up with
\begin{align*}
\int_b^{\infty} |\psi'(t)|^2 e^{t} \, \dd t + \alpha e^b |\psi(b)|^2 
&\ge \int_b^{\infty} |\varphi'(t)|^2 \, \dd t + \frac{1}{4} \int_b^{\infty} |\varphi(t)|^2 \, \dd t + \left(\alpha+\frac12\right) |\varphi(b)|^2 \\
&\ge \frac{1}{4} \int_b^{\infty} |\varphi(t)|^2 \, \dd t \\
&= \frac{1}{4} \int_b^{\infty} |\psi(t)|^2 e^{t} \, \dd t. \qedhere
\end{align*}
\end{proof}
\begin{Lemma}\label{auxlem1}
  Let $b\geq 0$, and $\alpha\geq -2^{-1}\tanh(b)$. Then
   \begin{align*}
       \inf_{\substack{\psi \in C_c^1([b,\infty)) \\ \psi \not= 0}} \frac{\displaystyle\int_b^{\infty} |\psi'(t)|^2 \cosh(t) \dd t+\alpha\cosh(b)|\psi(b)|^2}{\displaystyle\int_b^{\infty} |\psi(t)|^2 \cosh(t) \dd t}\geq \frac14.
   \end{align*}
\end{Lemma}
\begin{proof}
For any $\psi\in C^{1}_c([b,\infty))$  not identically zero, 
we set $\varphi(t)=\sqrt{\cosh(t)} \, \psi(t)$.
Following the lines of the proof of the previous lemma
we obtain
\begin{align*}
    \int_b^{\infty} |\psi'(t)|^2 \cosh(t) \dd t&= \int_b^{\infty} \left|\frac{\varphi'}{\sqrt{\cosh(t)}} - \frac{1}{2}\frac{\varphi \sinh(t)}{\cosh^{3/2}(t)}\right|^2\cosh(t) \dd t\\&=\int_b^{\infty} \left|\varphi' - \frac{1}{2}\varphi \tanh(t)\right|^2 \dd t\\&=\int_b^{\infty} |\varphi'|^2 \dd t - \frac{1}{2}\int_b^{\infty} (|\varphi|^2)' \tanh(t) \dd t + \frac{1}{4}\int_b^{\infty} |\varphi|^2 \tanh^2 (t) \dd t \\&=\int_{b}^{\infty}|\varphi'|^{2}\dd t+\frac{\tanh (b)|\varphi(b)|^{2}}{2}+\frac{1}{4}\int_{b}^{\infty}\left(\tanh^{2}(t)+\frac{2}{\cosh^{2}(t)}\right)|\varphi|^{2}\dd t.
\end{align*}
Using the identity for hyperbolic functions  $$\tanh^{2}(t)+\frac{2}{\cosh^{2}(t)}=1+\frac{1}{\cosh^{2}(t)},$$
we finally obtain
\[
\begin{aligned}
    \int_b^{\infty} |\psi'(t)|^2 \cosh(t) \dd t+\alpha\cosh(b)|\psi(b)|^2&\geq \frac14\int_b^{\infty}|\varphi|^2\dd t+\left(\alpha+\frac{\tanh(b)}{2} \right)|\varphi(b)|^2\\&\geq\frac14\int_b^{\infty}|\varphi|^2\dd t\\&=\frac14\int_b^{\infty}|\psi|^2\cosh(t)\dd t .\qedhere
\end{aligned}
\] 
\end{proof}
Now we are in position to state and prove a theorem on the sign of the critical coupling constant for exteriors of geodesically convex sets in the hyperbolic plane.
\begin{Theorem}\label{alphaubg}
    Assume that $\Omega\subset\mathbb{H}^2$ is  an open, bounded, geodesically convex set  with $C^2$-smooth boundary. Then, it follows that $$\alpha_{\star}(\Omg^\ext)\leq 0,$$ 
   with strict inequality whenever $\Omega$ is strictly geodesically convex.
\end{Theorem}
\begin{proof}
From the hypothesis on $\Omega$, the complement $\Omg^\ext$ is necessarily connected. We employ geodesic parallel coordinates
$$
\mathcal{L}:
\partial\Omega\times(0,\infty)
\ni (s,t)
\mapsto \exp_s(tn(s))\in\Omg^\ext,
$$
where $n$ denotes the outer unit normal to $\partial\Omega$. %
The existence and uniqueness of the metric projection onto $\Omega$ stated in Proposition \ref{lowboundreach} ensure that the map $\mathcal{L}$ is bijective. As a consequence of the smoothness assumptions, it follows that $\mathcal{L}\in C^1(\partial\Omega\times(0,\infty))$. Moreover, it is of class $C^\infty$ with respect to the variable $t$ for fixed $s$.
According to \cite[Section 2]{HA}, in parallel geodesic coordinates, 
the pullback metric expression is given by
$$\dd\mathcal{L}^2=J^2(s,t)
\dd s^2+\dd t^2.$$
The map $J\colon \p\Omg\times(0,\infty)\arr\dR_+$ is a continuous function\footnote{$\dR_+ = (0,\infty)$} satisfying
\begin{equation}\label{eq:J}
\left\{
\begin{aligned}
J_{tt}      &= J, \\
J(s,0)   &= 1, \\
J_t(s,0)   &= -\kappa(s),
\end{aligned}
\right.
\end{equation}
where $\kappa(s)$
denotes the geodesic curvature of $\partial\Omega$ at $s$ computed with respect to the outer unit normal. It is non-positive, since $\Omega$ is strongly convex \cite[Lemma 3.1]{GS}. By solving the system~\eqref{eq:J} we deduce that
\[
J(s,t) = \cosh t - \kappa(s)\, \sinh t.
\]
The Jacobian of the transformation $\mathcal{L}$ is given by $J$ and it is strictly positive, therefore $\mathcal{L}$ is a global diffeomorphism.
Thus, for any $\psi\in C^2_{\rm c}(\ov{\Omg^\ext})$ not identically zero,
\begin{align*}
    \frac{\displaystyle\int_{\Omg^\ext} |\nabla\psi|^2\dd\mu+\alpha\int_{\partial \Omega}|\psi|^2\dd \sigma}{\displaystyle\int_{\Omg^\ext} |\psi|^2 \dd\mu}&= \frac{\displaystyle\int_{\partial\Omega}\int_0^{\infty}|(\nabla\psi)(\mathcal{L}(s,t))|^2J(s,t)\dd t\dd\sigma(s)+\alpha\int_{\partial \Omega}|\psi(\mathcal{L}(s,0))|^2\dd \sigma}{\displaystyle\int_{\partial\Omega}\int_0^{\infty}|\psi(\mathcal{L}(s,t))|^2J(s,t)\dd t\dd\sigma(s)}\\&\geq\frac{\displaystyle\int_{\partial\Omega}\int_0^{\infty}\left|\frac{\p}{\p t}(\psi(\mathcal{L}(s,t))\right|^2J(s,t)\dd t\dd\sigma(s)+\alpha\int_{\partial \Omega}|\psi(\mathcal{L}(s,0))|^2\dd \sigma}{\displaystyle\int_{\partial\Omega}\int_0^{\infty}|\psi(\mathcal{L}(s,t))|^2J(s,t)\dd t\dd\sigma(s)},
\end{align*}
where in the second step we projected the gradient of $\psi$ to the vector $n(s)$ in the tangent plane to $\Omg^\ext$ at the point $\cL(s,t)$.
By setting $\eta(s,t)=\psi(\mathcal{L}(s,t)),$ we get
\begin{align}\label{ingpc}
    \frac{\displaystyle\int_{\Omg^\ext} |\nabla\psi|^2 \dd\mu+\alpha\int_{\partial \Omega}|\psi|^2\dd \sigma}{\displaystyle\int_{\Omg^\ext} |\psi|^2 \dd\mu}\geq \frac{\displaystyle\int_{\partial\Omega}\int_0^{\infty}\left|\frac{\partial \eta}{\partial t}(s,t)\right|^2J(s,t)\dd t\dd\sigma(s)+\alpha\int_{\partial \Omega}|\eta(s,0)|^2\dd \sigma}{\displaystyle\int_{\partial\Omega}\int_0^{\infty}|\eta(s,t)|^2J(s,t)\dd t\dd\sigma(s)}.
\end{align}
We split $\partial\Omega$ in three disjoint subsets $U_1, U_2$ and $U_3$, with $s\in U_1$ if and only if $-1<\kappa(s)\leq 0$; $s\in U_2$ if and only if $\kappa(s)=-1$; otherwise, $s\in U_3$. If $s\in U_1$, then there exists $z_1(s)\geq 0$ such that $$-\kappa(s)=\displaystyle\frac{\sinh(z_1(s))}{\cosh(z_1(s))},$$ which implies
$$\cosh(t)-\kappa(s)\sinh(t)=\frac{\cosh(t+z_1(s))}{\cosh(z_1(s))}.$$ If $s\in U_3$, then there exists $z_2(s)>0$ such that $$-\kappa(s)=\displaystyle\frac{\cosh(z_2(s))}{\sinh(z_2(s))},$$ which gives
$$\cosh(t)-\kappa(s)\sinh(t)=\frac{\sinh(t+z_2(s))}{\sinh(z_2(s))}.$$ Splitting in three integrals both numerator and denominator in \eqref{ingpc} and making respectively the substitutions $w=t+z_1(s)$, $w=t+z_2(s)$, we get
\begin{align*}
\int_{\partial\Omega}\int_0^{\infty}
\left|\frac{\partial \eta}{\partial t}(s,t)\right|^2
J(s,t)\,\dd t\,\dd\sigma(s)
&=\int_{U_1}\int_{z_1(s)}^{\infty}
\left|\frac{\partial \eta}{\partial t}(s,w-z_1(s))\right|^2
\cosh(w)\,\dd w\,\frac{\dd\sigma(s)}{\cosh(z_1(s))}\\
&\quad+\int_{U_2}\int_{0}^{\infty}
\left|\frac{\partial \eta}{\partial t}(s,t)\right|^2
e^t\,\dd t\,\dd\sigma(s)\\
&\quad+\int_{U_3}\int_{z_2(s)}^{\infty}
\left|\frac{\partial \eta}{\partial t}(s,w-z_2(s))\right|^2
\sinh(w)\,\dd w\,\frac{\dd\sigma(s)}{\sinh(z_2(s))},
\\[1ex]
\alpha\int_{\partial \Omega}|\eta(s,0)|^2\dd \sigma=&\int_{U_1}|\eta(s, z_1(s)-z_1(s))|^2\dd \sigma+\int_{U_2}|\eta(s, 0)|^2\dd \sigma\\&+\int_{U_3}|\eta(s, z_2(s)-z_2(s))|^2\dd \sigma
\\[1ex]
\int_{\partial\Omega}\int_0^{\infty}
|\eta(s,t)|^2J(s,t)\,\dd t\,\dd\sigma(s)
&=\int_{U_1}\int_{z_1(s)}^{\infty}
|\eta(s,w-z_1(s))|^2\cosh(w)\,\dd w\,
\frac{\dd\sigma(s)}{\cosh(z_1(s))}\\
&\quad+\int_{U_2}\int_{0}^{\infty}
|\eta(s,t)|^2e^t\,\dd t\,\dd\sigma(s)\\
&\quad+\int_{U_3}\int_{z_2(s)}^{\infty}
|\eta(s,w-z_2(s))|^2\sinh(w)\,\dd w\,
\frac{\dd\sigma(s)}{\sinh(z_2(s))}.
\end{align*}
We denote respectively by $N_1, N_2, N_3, B_1, B_2, B_3, D_1, D_2, D_3$ the nine integral terms above (in the order they appear in the text).
Given $s\in U_1$, we set $\varphi_1(w)=\eta(s,w-z_1(s))$. If $\alpha\geq \alpha_1:=-2^{-1}\inf_{s\in U_1}\tanh(z_1(s))$ then, according to Lemma \ref{auxlem1}, it follows that 
$$\int_{z_1(s)}^{\infty}\left|\varphi_1'(w)\right|^2\cosh(w)
\dd w+\alpha\cosh(z_1(s))|\varphi_1(z_1(s))|^2\geq \frac14\int_{z_1(s)}^{\infty}\left|\varphi_1(w)\right|^2\cosh(w)\dd w,\quad s\in U_1,$$ 
which implies
$$N_1+B_1\geq \frac14 D_1.$$
If $\Omega$ is geodesically strictly convex, then by the Gauss formula (see, e.g., \cite[Theorem 8.2]{L18}), the geodesic curvature $\kappa(s)$ of $\partial\Omega$ is strictly negative. Since $\partial\Omega$ is $C^2$-smooth, $\kp(s)$ is continuous and therefore bounded away from zero. Consequently $\inf_{s\in U_1}z_1(s)>0$, which implies $\alpha_1<0$. When convexity is not strict, it follows that $\alpha_1 = 0$.\\
Arguing in a similar way, if we assume respectively that $\alpha\geq \alpha_2:=-2^{-1}$ and $\alpha\geq \alpha_3:=2^{-1}\sup_{s\in U_3}(\frac{1}{z_2(s)}-\coth(z_2(s)))$, then Lemmata ~\ref{auxlem2} and ~\ref{auxlem3} ensure that
$$N_2+B_2\geq \frac14 D_2,\qquad N_3+B_3\geq \frac14 D_3.$$
Since the geodesic curvature $\kp(s)$ of $\partial\Omega$ is bounded, then $\inf_{s\in U_3}z_2(s)>0$, which yields $\alpha_3<0$.

Let us fix $\aa \ge \alpha_0:=\max_{1\leq i\leq 3}\alpha_i$. Observe that $\alpha_0 < 0$ when $\Omega$ is strictly convex; otherwise, $\alpha_0 = 0$.
Then for any $\psi\in C^2_c(\overline{\Omg^\ext})$  not identically zero, we get using the preceding analysis that
$$\frac{\displaystyle\int_{\Omg^\ext}|\nabla\psi|^2\dd\mu+\alpha\int_{\partial\Omega}|\psi|^2\dd\sigma}{\displaystyle\int_{\Omg^\ext} |\psi|^2 \dd\mu} \geq\frac14,$$
by which the claim is proved in view of density of $C^2_c(\ov{\Omg^\ext})$ in $W^{1,2}(\Omg^\ext)$ (see Lemma~\ref{lem:density}).

\end{proof}
\begin{Remark}\label{rem:nonconvex}   
On the contrary to Theorem~\ref{alphaubg}, for some non-convex simply-connected domains in the hyperbolic plane we conjecture that
it can hold $\aa_\star(\Omg^\ext) > 0$. The example supporting this conjecture can be constructed as follows. Consider first the exterior of a geodesic annulus $A_{r,R}\subset\dH^2$ centred at the origin with radii $
R>r>0$ and define the Robin Laplacian on $A_{r,R}^\ext$. This operator can be naturally decomposed into orthogonal sum of two Robin Laplacians, one of which is defined 
on the geodesic disk
$B_r$ while the other is defined on $B_R^\ext$. We can ensure that for a sufficiently small positive Robin parameter $\aa > 0$ the component of the operator on $B_r$ has an eigenvalue below $\frac14$. In order to see that
it suffices to take a constant function as the trial function. Digging a small thin channel between $B_r$ and $B_R^\ext$ we can construct a connected exterior domain such that the Robin Laplacian with a sufficiently small positive Robin parameter has an eigenvalue below $\frac14$. The final step in the construction can be justified by showing that in the limit of a thin channel there holds a generalised norm resolvent convergence to the decoupled case. This analysis goes beyond the scope of the present paper.  
\end{Remark}

\section{The spectral problem in the exterior of a disk}
%
\label{sec:disk}
In this section, we establish some properties of 
$\lambda_1^\alpha(B_R^\mathrm{ext})$,
where $\alpha$ is negative and $B_R$ is the geodesic disk of radius $R>0$ in $\mathbb{H}^2$.
\subsection{Separation of variables}
 Without loss of generality, we can assume that 
the centre of $B_R$ has coordinates $(1,0,0)$ in the hyperboloid model of $\dH^2$. By employing geodesic polar coordinates,
we introduce the following complete family of mutually orthogonal projections
\begin{align*}
    \Pi_n\colon L^2(B_R^\mathrm{ext})\rightarrow L^2(B_R^\mathrm{ext}),\qquad (\Pi_n u)(r,\theta):=\frac{e^{\ii n\theta}}{2\pi}\int_0^{2\pi}u(r,\theta')e^{-\ii n\theta'}\dd\theta',\quad n\in\mathbb{Z},
\end{align*}
and the unitary mappings
\begin{align*}
   \sfU_n\colon \ran\Pi_n\rightarrow L^2((R,\infty);\sinh r\dd r),\qquad (\sfU_n u)(r)=\frac{1}{\sqrt{2\pi}}\int_0^{2\pi}u(r,\theta')e^{-\ii n\theta'}\dd\theta',\quad n\in\mathbb{Z}.
\end{align*}
Therefore, it is possible to write $L^2(B_R^\mathrm{ext})$ as an orthogonal sum
\begin{align*}
    L^2(B_R^\mathrm{ext})=\bigoplus_{n\in\mathbb{Z}}\text{ran }\Pi_n\simeq \bigoplus_{n\in\mathbb{Z}} L^2((R,\infty);\sinh r\dd r).
\end{align*}
With the help of \cite[Proposition 1.15]{Sm}, it can be proved that each summand in the above orthogonal sum is a reducing subspace for the operator $-\Delta^{B_R^{\text{ext}}}_{\alpha}$. This induces an orthogonal decomposition of the Robin Laplacian 
\begin{align*}
    -\Delta^{B_R^{\text{ext}}}_{\alpha}\simeq\bigoplus_{n\in\mathbb{Z}}(\sfU_n^{-1}\sfH_{\alpha, R, n}\sfU_n),
\end{align*}
where the self-adjoint fibre operators $\sfH_{\alpha, R, n}$ in $L^2((R,\infty);\sinh r\dd r)$ are associated with the closed, symmetric, densely defined and semi-bounded quadratic forms
\begin{align*}
    \frh_{\alpha, R, n}[f]:&=\int_R^{\infty}\left(\abs{f'(r)}^2+\frac{n^2}{\sinh^2{r}}\abs{f(r)}^2\right)\sinh{r}\dd r+\alpha \sinh{R}\, \abs{f(R)}^2,\\
    \dom\frh_{\alpha, R, n}:&=\big\{f\in L^2((R,\infty);\sinh r\dd r)\,\colon\, \sfU_n^{-1}f\in W^{1,2}(B_R^{\text{ext}})\}\\&=\{f\,\colon\, f,f'\in L^2((R,\infty);\sinh r\dd r)\}.
    \end{align*}
The action and the domain of the fibre operators associated with the above sesquilinear form can be characterised as
\begin{align*}
    \sfH_{\alpha, R, n}&=-f''(r)-\coth(r) f'(r)+\frac{n^2}{\sinh^2(r)}f(r),\\
     \dom \sfH_{\alpha, R, n}&=\left\{f\,\colon\, f, f''+\coth(r)f'\in L^2((R,\infty);\sinh r \dd r), f'(R)=\alpha f(R)\right\}.
\end{align*}
The form $\frh_{\alpha, R, 0}$ is the smallest in the sense of the ordering of forms and from the min-max principle, we get that the lowest spectral point of $\sfH_{\alpha, R, 0}$ is not larger than that of $ \sfH_{\alpha, R, n}$, for any $n\not =0$.
Moreover, if the lowest spectral point of $\sfH_{\alpha, R, 0}$
is a discrete eigenvalue, then by the min-max principle it is strictly smaller than the lowest spectral point of $ \sfH_{\alpha, R, n}$ for any $n\not =0$.  By construction, the lowest spectral point of $-\Delta^{B_R^{\text{ext}}}_{\alpha}$ is the same of the fibre operator $ \sfH_{\alpha, R, 0}$ and moreover, the corresponding eigenfunction is radial provided this lowest spectral point is a discrete eigenvalue.
\subsection{Explicit expression of the first eigenfunction}
\label{ssec:first}
In order to express the eigenfunctions of $\sfH_{\aa,R,0}$
in terms of special functions we analyse the following ordinary differential equation
\begin{equation}\label{pdeball}
- y''(x) - \coth(x)\, y'(x) = k\, y(x), \qquad x \in [R, \infty), \quad k < \frac{1}{4}.
\end{equation}
By performing the substitution
\begin{equation*}
t = \cosh(x),
\end{equation*}
we obtain a differential equation on $y$ in terms of variable $t$
\[
- (t^2 - 1) \frac{\dd^2 y}{\dd t^2} - 2t \frac{\dd y}{\dd t} = k y.
\]
This second-order linear differential equation is of the form of the Legendre equation. The general solution is given by
\[
y(t) = A P^0_{\nu}(t) + B Q^0_{\nu}(t),
\]
where $ P^0_{\nu}(t)$, $Q^0_{\nu}(t)$ are respectively the Legendre functions of the first and second kind (see \cite[Chapter 14]{NIST}) with $\mu=0$ and $\nu$ satisfying
\[
\nu(\nu + 1) = -k,
\]
which gives
\[
\nu^2 + \nu + k = 0 \Rightarrow \nu = \frac{-1 \pm \sqrt{1 - 4k}}{2}.
\]
In order to get two linearly independent solutions \cite[Subsection 14.2(iii)]{NIST}, we can choose the value 
\begin{align}\label{nu}
    \nu = \frac{-1 + \sqrt{1 - 4k}}{2}>-\frac{1}{2}.
\end{align}
In conclusion, the general solution of the original equation is
\[
y(x) = A P^0_{\nu}(\cosh(x)) + BQ^0_{\nu}(\cosh(x)), \quad \text{with } \nu = \frac{-1 + \sqrt{1 - 4k}}{2}.
\]
At this point, it is necessary to analyse the $L^2$-summability of the solution
with respect to the measure $\sinh(x)\dd x$ on the interval $[R, \infty)$.
To this aim, we recall some asymptotic results. Clearly
\[
\cosh x \sim \frac{e^x}{2},\qquad \sinh x \sim \frac{e^x}{2},
\] 
for large $x$.
According to \cite[Formulas 14.3.10, 14.8.12, 14.8.15]{NIST}, we have in the limits $z,x\arr\infty$
\[
\begin{aligned}
P^0_{\nu}(z) &\sim \frac{\Gamma\left(\nu + \frac{1}{2}\right)}{{\pi}^{\frac{1}{2}} \, \Gamma(\nu + 1)} (2z)^{\nu} \qquad &\Longrightarrow& \qquad P^0_{\nu}(\cosh x) \sim \frac{\Gamma\left(\nu + \frac{1}{2}\right)}{{\pi}^{\frac{1}{2}}\Gamma(\nu+1)}e^{\nu x},\\
Q^0_{\nu}(z) &\sim \frac{\pi^{1/2}\Gamma(\nu + 1)}{\Gamma(\nu + \frac{3}{2}) (2z)^{\nu + 1}} \qquad &\Longrightarrow& \qquad Q^0_{\nu}(\cosh x) \sim 
\frac{\pi^{1/2}\Gamma(\nu + 1)}{\Gamma(\nu + \frac{3}{2})}e^{-(\nu+1) x}.
\end{aligned}
\]
Combining the asymptotic estimates, it follows that
in the limit $x\arr\infty$
\[
\begin{aligned}
|P^0_{\nu}(\cosh x)|^2 \sinh(x) &\sim \frac12\left(\frac{\Gamma\left(\nu + \frac{1}{2}\right)}{{\pi}^{\frac{1}{2}}\Gamma(\nu+1)}\right)^2 e^{(2 \nu + 1) x},\\
|Q^0_{\nu}(\cosh x)|^2 \sinh(x) &\sim\frac12\left(\frac{\pi^{1/2}\Gamma(\nu + 1)}{\Gamma(\nu + \frac{3}{2})}\right)^2 e^{(-2 \nu - 1) x}.
\end{aligned}
\]
Since we are looking for square-summable eigenfunctions, the choice of $\nu$ imposes  $A=0$. Finally, the expression of the original solution simplifies to 
\begin{equation}\label{eq:y}
	y(x)=BQ^0_{\nu}(\cosh(x)),
\end{equation}
with $B\not=0$ for non-trivial solutions.

We will make use of the following properties of the Legendre function of second kind. The function $Q^0_{\nu}$ admits an integral representation \cite[Formulas 14.3.10, 14.12.6]{NIST}
\begin{align}\label{eq:Qidentity}
    Q_{\nu}^{0}(x)=\frac{\pi^{1/2}}{\Gamma\left(\frac{1}{2}\right)}\int_{0}^{\infty}\frac{1}{\left(x+\sqrt{x^{2}-1}\cosh t\right)^{\nu+1}}\dd t,\qquad \text{for }\quad x>1,\, \nu+1>0,
\end{align}
from which it is clear that it is a positive, strictly decreasing, smooth function. Moreover,
\begin{align}\label{ratio}
    \frac{ Q_{\nu+1}^{0}(x)}{ Q_{\nu}^{0}(x)}\leq \frac{1}{x+\sqrt{x^2-1}},\qquad \text{for }x>1.
\end{align}
Derivatives can be computed using recurrence relations \cite[Formula 14.10.4]{NIST}
\begin{align}\label{rec}
    (1 - x^2) \frac{\dd}{\dd x} Q^{0}_\nu(x) = (\nu + 1) x Q^{0}_\nu(x) - (\nu + 1) Q^{0}_{\nu + 1}(x).
\end{align}

\subsection{The critical Robin boundary parameter}

\noindent Next, we provide a sharper upper bound for the critical value $\alpha_{\star}(B_R^{\text{ext}})$, improving upon the one given in Theorem \ref{alphaubg} in the case of the geodesic disk.

\begin{Lemma}\label{alphaub}
For every $R>0$, it holds
$$\alpha_{\star}(B_R^{\ext})\leq \frac{1}{2}(e^{-R}-\coth(R)).$$
\end{Lemma}
\begin{proof}
Assume that $\aa \le 0$ is such that the operator $-\Delta_\aa^{B^\ext_R}$ has a discrete eigenvalue below the bottom of its essential spectrum. The eigenfunction corresponding to its lowest eigenvalue is radially symmetric and as a function of the radial variable can be identified with the function given by identity~\eqref{eq:y}.
   Starting from the boundary condition
\[
-y'(R) + \alpha y(R) = 0,
\]
and applying, respectively, relations \eqref{rec}, \eqref{nu}, and \eqref{ratio}, we obtain
\begin{align*}
    \alpha &=-(\nu+1)\left[\frac{\cosh(R)Q^{0}_\nu(\cosh(R))-Q^{0}_{\nu+1}(\cosh(R))}{\sinh(R)Q^{0}_\nu(\cosh(R))}\right]\\&=(\nu+1)\left[\frac{Q^{0}_{\nu+1}(\cosh(R))}{Q^{0}_{\nu}(\cosh(R))}-\coth(R)\right]\\
    &\leq (\nu+1)\left[e^{-R}-\coth(R)\right]\\
    &\leq \frac{1}{2}\left[e^{-R}-\coth(R)\right],
\end{align*}
where in the last step we used that $e^{-R}-\coth(R)<0$
and that $\nu > -\frac12$.
\end{proof}
\begin{Remark}
 The given upper bound is a negative function. Moreover, it diverges to $-\infty$ as $R\to 0$ and converges to $-\frac12$ as $R\to+\infty.$ 
Direct consequences are
$$\limsup_{R\to \infty}\,\alpha_{\star}(B_R^{\text{ext}})\le -\frac12,\qquad \alpha_{\star}(B_R^{\text{ext}})\,\,\xrightarrow{R \to 0}\,\,-\infty.$$ 
\end{Remark}
\begin{Remark}\label{alphalb}
    We observe that if  $\lambda_1^\aa(B_R^{\text{ext}})>0$, then $\frac12<\nu+1<1$. This imposes a constraint on the parameter $\alpha$, allowing us to derive a lower bound for it
    \begin{align*}
        \alpha =(\nu+1)\left[\frac{Q^{0}_{\nu+1}(\cosh(R))}{Q^{0}_{\nu}(\cosh(R))}-\coth(R)\right]> -(\nu+1)\coth(R)>-\coth(R).
    \end{align*}
\end{Remark}
\subsection{The monotonicity of the lowest eigenvalue with respect to radius}
In this subsection, we establish that the lowest eigenvalue of the Robin Laplacian in the exterior of the geodesic disk in the hyperbolic plane is a decreasing function of the radius. This property was previously observed in the case of the Robin Laplacian in the Euclidean plane~\cite[Proposition~5]{KL18}. The case of the hyperbolic plane is more subtle, because the lowest eigenvalue can also be positive for moderate negative Robin parameters.
\begin{Proposition}[Monotonicity property]\label{mpeig}
Assume that $0<R<\overline{R}$ and that $\alpha<\alpha_{\star}(B_R^\ext)$. Then
\begin{equation}\label{spinb}
    \lambda_{1}^{\alpha}(B_{R}^\ext)>\lambda_{1}^{\alpha}(B_{\overline{R}}^\ext).
\end{equation}
\end{Proposition}

\begin{proof}
\emph{Step 1: variational characterisation of the lowest spectral point.}
From the analysis in previous subsections of this section, we have
the following variational characterisation of the lowest eigenvalue of $-\Delta_\aa^{B_R^\ext}$
\begin{align*}
    \lambda_{1}^{\alpha}(B_R^\ext)&=\inf_{\begin{smallmatrix}\psi\in W^{1,2}([R,\infty),\sinh(t)\dd t)\\ \psi \ne 0 \end{smallmatrix}}
		\frac{\displaystyle\int_{R}^{\infty} |\psi'(t)|^2 \sinh(t)\dd t+ \alpha  \sinh(R) |\psi(R)|^2}{\displaystyle\int_R^{\infty}|\psi(t)|^2\sinh(t)\dd t}\\&=\inf_{\begin{smallmatrix}\psi\in W^{1,2}([0,\infty),\sinh(t)\dd t)\\ \psi \ne 0 \end{smallmatrix}}
		\frac{\displaystyle\int_{0}^{\infty} |\psi'(t)|^2 \sinh(t+R) \dd t+ \alpha \, \sinh(R) |\psi(0)|^2}{\displaystyle\int_0^{\infty}|\psi(t)|^2\sinh(R+t) \dd t}.
\end{align*}
The infimum is achieved by any eigenfunction corresponding to the lowest eigenvalue. In particular, by the analysis in Subsection~\ref{ssec:first}, we can choose the ground state in the form $\varphi(t)=Q^0_{\nu}(\cosh(t+R))$ with $\nu$ related to the lowest eigenvalue by the identity~\eqref{nu} with $k = \lm_1^\aa(B_R^\ext)$.  

\smallskip

\noindent\emph{Step 2: monotonicity and convexity of $\varphi$.}
Differentiating the integral representation~\eqref{eq:Qidentity} we obtain
\begin{align*}
    (Q_{\nu}^{0})^{\prime}(x)&=-\frac{\pi^{1/2}}{\Gamma(\frac{1}{2})}\int_{0}^{\infty}\frac{(\nu+1)\left(1+\frac{x}{\sqrt{x^{2}-1}}\cosh(z)\right)}{\left(x+\sqrt{x^{2}-1}\cosh(z)\right)^{\nu+2}}\,\dd z.
\end{align*}
Using the above formula we compute
{\small
\[
\begin{aligned}
	\varphi'(t)&=-\frac{\pi^{1/2}(\nu+1)}{\Gamma(\frac{1}{2})}\int_{0}^{\infty}\frac{\left(\sinh(t+R)+\cosh(t+R)\cosh(z)\right)}{\left(\cosh(t+R)+\sinh(t+R)\cosh(z)\right)^{\nu+2}}\dd z <0,\\\varphi''(t)&=\frac{\pi^{1/2}(\nu+1)}{\Gamma(\frac{1}{2})}\int_{0}^{\infty}
	\frac{\cosh^2(t+R)\big[(2 + \nu) (\cosh(z) + \tanh(R + t))^2-(1 + \cosh(z) \tanh(R + t))^2\big]}{\left(\cosh(t+R)+\sinh(t+R)\cosh(z)\right)^{\nu+3}}\,\dd z.
\end{aligned}
\]}
In order to determine the sign of $\varphi''$, 
it is enough to study the function
\begin{align*}
    f(t):=
    \sqrt{\nu+2}(\cosh(z) + \tanh(R + t))-(1 + \cosh(z) \tanh(R + t)).
\end{align*}
Since $\sqrt{\nu+2}>\sqrt{\frac{3}{2}}>1$ and $0\leq \tanh(t+R)<1$, we get
$$
  f(t) > \sqrt{\frac{3}{2}} \tanh(R + t) - 1
+ \cosh(z)\left( \sqrt{\frac{3}{2}} - \tanh(R + t) \right).
$$
By setting $w := \tanh(R+t)$, 
the linear function $$A(w)=\sqrt{\frac{3}{2}} w - 1
+ \cosh(z)\left( \sqrt{\frac{3}{2}} - w \right)$$ satisfies $$A(0)=- 1+ \sqrt{\frac{3}{2}}\cosh(z)>0,\qquad A(1)=\sqrt{\frac{3}{2}} - 1
+ \cosh(z)\left( \sqrt{\frac{3}{2}} - 1\right)>0,$$ 
therefore
$f(t) > 0$,
from which it follows $\varphi''(t)>0$.

\smallskip

\noindent \emph{Step 3: a sufficient condition for monotonicity of the eigenvalue.}
Given $s\geq 1$, we have
\begin{align*}
     \lambda_{1}^{\alpha}(B_{sR}^{\text{ext}})&\leq\frac{\displaystyle\int_{0}^{\infty} |\varphi'(t)|^2 \frac{\sinh(t+sR)}{\sinh(sR)} \dd t+ \alpha |\varphi(0)|^2}{\displaystyle\int_0^{\infty}|\varphi(t)|^2\frac{\sinh(t+sR)}{\sinh(sR)}\dd t}\\&=\frac{\coth(sR)\displaystyle\int_{0}^{\infty} |\varphi'(t)|^2 \sinh(t) \dd t+\displaystyle\int_{0}^{\infty} |\varphi'(t)|^2 \cosh(t) \dd t +\alpha  |\varphi(0)|^2}{\coth(sR)\displaystyle\int_{0}^{\infty} |\varphi(t)|^2 \sinh(t) \dd t+\displaystyle\int_{0}^{\infty} |\varphi(t)|^2 \cosh(t) \dd t }.
\end{align*}
We denote by $g(s)$ the last expression and we define the following quantities
\begin{align*}
    a&:=\displaystyle\int_{0}^{\infty} |\varphi'(t)|^2 \sinh(t) \dd t;\\
    b&:=\displaystyle\int_{0}^{\infty} |\varphi'(t)|^2 \cosh(t) \dd t +\alpha  |\varphi(0)|^2;\\
    c&:=\displaystyle\int_{0}^{\infty} |\varphi(t)|^2 \sinh(t) \dd t;\\
    d&:=\int_{0}^{\infty} |\varphi(t)|^2 \cosh(t) \dd t.
\end{align*}
Then
$$g'(s)=R\dfrac{bc-ad}{\left(c \cosh(sR)+d\sinh(sR)\right)^2}.$$
The proof ends by showing that $g'(s)<0$ for $s\geq 1,$ since in this case
$$\lambda_{1}^{\alpha}(B_{sR}^{\text{ext}})\leq g(s)< g(1)=\lambda_{1}^{\alpha}(B_{R}^{\text{ext}}),\quad\text{for }s>1.$$
We have
\begin{align*}
    g'(s)<0\quad\iff\quad\frac{\displaystyle\int_{0}^{\infty} |\varphi'(t)|^2 \cosh(t) \dd t +\alpha  |\varphi(0)|^2}{\displaystyle\int_{0}^{\infty} |\varphi(t)|^2 \cosh(t) \dd t}<\dfrac{\displaystyle\int_{0}^{\infty} |\varphi'(t)|^2 \sinh(t)\dd t}{\displaystyle\int_{0}^{\infty} |\varphi(t)|^2 \sinh(t) \dd t}. 
\end{align*}

\smallskip

\noindent\emph{Step 4: a verification of the sufficient condition for monotonicity.}
It is not restrictive to assume
$$\displaystyle\int_{0}^{\infty} |\varphi'(t)|^2 \cosh(t) \dd t +\alpha  |\varphi(0)|^2>0,$$
otherwise the proof would be complete. Employing strict monotonicity 
and strict convexity of~$\varphi$ we get
\begin{equation*}
\begin{aligned}
    &\frac{\displaystyle\int_{0}^{\infty} |\varphi'(t)|^2 \cosh(t) \dd t + \alpha  |\varphi(0)|^2}{\displaystyle\int_{0}^{\infty} |\varphi(t)|^2 \cosh(t) \dd t} 
    \leq \frac{\displaystyle\int_{0}^{\infty} |\varphi'(t)|^2 \cosh(t) \dd t + \alpha  |\varphi(0)|^2}{\displaystyle\int_{0}^{\infty} |\varphi(t)|^2 \sinh(t) \dd t} \\
    &= \frac{\displaystyle\int_{0}^{\infty} |\varphi'(t)|^2 \sinh(t) \dd t + \int_{0}^{\infty} |\varphi'(t)|^2 e^{-t} \dd t + \alpha \, |\varphi(0)|^2}{\displaystyle\int_{0}^{\infty} |\varphi(t)|^2 \sinh(t) \dd t} \\&< \frac{\displaystyle\int_{0}^{\infty} |\varphi'(t)|^2 \sinh(t) \dd t + \varphi'(0) \int_{0}^{\infty} \varphi'(t) e^{-t} \dd t + \alpha  |\varphi(0)|^2}{\displaystyle\int_{0}^{\infty} |\varphi(t)|^2 \sinh(t) \dd t}\\
    &< \frac{\displaystyle\int_{0}^{\infty} |\varphi'(t)|^2 \sinh(t) \dd t + \varphi'(0) \int_{0}^{\infty} \varphi'(t) \dd t + \alpha \, |\varphi(0)|^2}{\displaystyle\int_{0}^{\infty} |\varphi(t)|^2 \sinh(t) \dd t} \\
&= \frac{\displaystyle\int_{0}^{\infty} |\varphi'(t)|^2 \sinh(t) \dd t}{\displaystyle\int_{0}^{\infty} |\varphi(t)|^2 \sinh(t) \dd t},
\end{aligned}
\end{equation*}
where in the last step we used the boundary condition $-\varphi'(0)+\alpha\varphi(0)=0.$
\end{proof}
\begin{Remark}
An immediate consequence is that, if $0 < R < \overline{R}$, then $\alpha_{\star}(B_R^{\text{ext}}) \leq \alpha_{\star}(B_{\overline{R}}^{\text{ext}})$. Clearly, inequality~\eqref{spinb} remains valid for $\alpha$ such that $\alpha_{\star}(B_{\overline{R}}^{\text{ext}}) > \alpha \geq \alpha_{\star}(B_R^{\text{ext}})$, while equality holds if $\alpha \geq \alpha_{\star}(B_{\overline{R}}^{\text{ext}})$.
\end{Remark}

\section{Main result}
%
\label{sec:main}
In this section, we will state and prove the main result of this paper on
comparison between the lowest Robin eigenvalue in the exterior of a geodesically convex bounded open set in the hyperbolic plane and the lowest Robin eigenvalue in the exterior of a geodesic disk. We also remark that this theorem generalises the monotonicity property in Proposition~\ref{mpeig}.
\begin{Theorem}\label{thm:main}
Let $\Omega\subset\mathbb{H}^2$ be an open, bounded and geodesically convex set. Let $B_R\subset\dH^2$ be a geodesic disk of radius $R$ such that  $$\coth(R)\geq \frac{|\Omega| + 2\pi}{\mathcal{H}^{1}(\partial\Omega)},$$
then, for any $\alpha\in\mathbb{R}$, it holds
$$\lambda_1^{\alpha}(\Omega^\ext)\leq \lambda_1^{\alpha}(B_R^\ext).$$
\end{Theorem}
\begin{proof}
If $\alpha\geq \aa_\star(B_R^\ext)$, then the claim easily follows from the fact that $$\lambda_1^{\alpha}(B_R^{\text{ext}})=\frac14.$$ We now consider the case $\alpha< \aa_\star(B_R^\ext)$. Let $\rho\colon \Omg^\ext\mapsto (0,\infty)$ be the hyperbolic distance
function from the boundary of $\Omega$. 
Furthermore, we define the auxiliary function by
\[
\begin{aligned}
	&L \colon \dR_+ \arr \dR_+ ,\qquad 
	&L(r)& := \mathcal{H}^1\left(\{x\in \Omg^\ext \,\colon\, \rho(x) = r\}\right).
\end{aligned}
\]
%

Note that $L(r)$ can be understood for every $r>0$ as the perimeter of the outer parallel body $\Omega_r$, according to Theorem \ref{Theorem: steiner}.
The Steiner's formula guarantees that $ L(r) = \mathcal{H}^1(\partial \Omega_r) $ for almost all $r$.
Given a  function $\psi\in C^\infty_0([0,\infty))$, we define the function 
$u := \psi \circ \rho\colon \overline{\Omg^\ext}\mapsto\mathbb{C}$. 
Since $\rho$ is Lipschitz continuous, the function $u$ is Lipschitz continuous and compactly supported in $\overline{\Omg^\ext}$. Thus, we have $u \in W^{1,2}(\Omega^{\mathrm{ext}})$, and by applying the co-area formula, we obtain
%
\begin{align*}
   \|\nabla u\|^2_{L^2(\Omg^\ext)} 
	& 
	= 
	\int_0^{\infty} |\psi'(t)|^2 L(t)  \dd t\,,\\
    \|u\|^2_{L^2(\Omega^\ext)} & 
	= 
	\int_0^{\infty} |\psi(t)|^2 L(t) \dd t\,,\\
    \| u \|^2_{L^2(\partial\Omega^\ext)} & 
	= 
	\cH^1(\p\Omg) \, |\psi(0)|^2 \,.
\end{align*}
In the following, we denote by $\varphi\colon\dR_+\arr\dR$ the ground state of $-\Delta_\aa^{B_R^\ext}$ as function of the distance to the boundary of the geodesic disk. An explicit representation of this function is given in Step~1 of the proof of Proposition~\ref{mpeig}.
By the min-max principle applied
for the quadratic form $Q^{\Omg^\ext}_\alpha$
on the subspace $\{\psi\circ\rho\colon \psi\in C^\infty_0([0,\infty))\}\subset W^{1,2}(\Omg^\ext)$, we get
using the identity in Theorem~\ref{thm:nonsmooth}
\begin{align*}
	\lambda_1^{\alpha}(\Omg^\ext)	
	&
	\le 
	\inf_{\begin{smallmatrix}\psi\in C^\infty_0([0,\infty))\\ \psi \ne 0 \end{smallmatrix}}
		\frac{\displaystyle\int_{0}^{\infty} |\psi'(t)|^2 L(t) \dd t+ \alpha  \mathcal{H}^{1}(\partial \Omega) |\psi(0)|^2}{\displaystyle\int_0^{\infty}|\psi(t)|^2 L(t)\dd t}\\&=\inf_{\begin{smallmatrix}\psi\in C^\infty_0([0,\infty))\\ \psi \ne 0 \end{smallmatrix}}
		\frac{\displaystyle\int_{0}^{\infty} |\psi'(t)|^2 \left(\frac{2\pi+|\Omega|}{\mathcal{H}^1(\partial\Omega)}\sinh(t)+\cosh(t)\right)  \dd t+ \alpha |\psi(0)|^2}{\displaystyle\int_0^{\infty}|\psi(t)|^2 \left(\frac{2\pi+|\Omega|}{\mathcal{H}^1(\partial\Omega)}\sinh(t)+\cosh(t)\right) \dd t}
        \\&\leq \frac{\displaystyle\int_{0}^{\infty} |\varphi'(t)|^2 \left(\frac{2\pi+|\Omega|}{\mathcal{H}^1(\partial\Omega)}\sinh(t)+\cosh(t)\right)  \dd t+ \alpha |\varphi(0)|^2}{\displaystyle\int_0^{\infty}|\varphi(t)|^2 \left(\frac{2\pi+|\Omega|}{\mathcal{H}^1(\partial\Omega)}\sinh(t)+\cosh(t)\right) \dd t}
        \\&\leq \frac{\displaystyle\int_{0}^{\infty} |\varphi'(t)|^2 \left(\frac{2\pi+|B_R|}{\mathcal{H}^1(\partial B_R)}\sinh(t)+\cosh(t)\right)  \dd t+ \alpha |\varphi(0)|^2}{\displaystyle\int_0^{\infty}|\varphi(t)|^2 \left(\frac{2\pi+|B_R|}{\mathcal{H}^1(\partial B_R)}\sinh(t)+\cosh(t)\right) \dd t}
        \\&=\inf_{\begin{smallmatrix}\psi\in C^\infty_0([0,\infty))\\ \psi \ne 0 \end{smallmatrix}}
		\frac{\displaystyle\int_{0}^{\infty} |\psi'(t)|^2 \left(\frac{2\pi+|B_R|}{\mathcal{H}^1(\partial B_R)}\sinh(t)+\cosh(t)\right)  \dd t+ \alpha |\psi(0)|^2}{\displaystyle\int_0^{\infty}|\psi(t)|^2 \left(\frac{2\pi+|B_R|}{\mathcal{H}^1(\partial B_R)}\sinh(t)+\cosh(t)\right) \dd t}\\
	&= \inf_{\begin{smallmatrix}\psi\in C^\infty_0([0,\infty))\\ \psi \ne 0 \end{smallmatrix}}
	\frac{\displaystyle\int_{0}^{\infty} |\psi'(t)|^2 \mathcal{H}^{1}(\partial B_{R+t}) \dd t+ \alpha  \mathcal{H}^{1}(\partial B_R)  |\psi(0)|^2}{\displaystyle\int_{0}^{\infty}|\psi(t)|^2
   \mathcal{H}^{1}(\partial B_{R+t}) \dd t} 
        \\ 
        &= 
	 \lambda_1^{\alpha}(B_R^{\text{ext}})\,,
\end{align*}	
where we have, respectively, used the identity
$$ \coth(R)= \frac{|B_R| + 2\pi}{\mathcal{H}^{1}(\partial B_R)}, $$
the fact that the infimum in the variational characterisation of $\lm_1^\aa(B_R^\ext)$ (given in Step 1 of Proposition~\ref{mpeig})
is attained on the function $\varphi$, and the strict monotonic increase of the function
$$ g(s)=\frac{\displaystyle\int_{0}^{\infty} |\varphi'(t)|^2 \bigl(s\sinh(t)+\cosh(t)\bigr)\dd t
+ \alpha |\varphi(0)|^2}
{\displaystyle\int_{0}^{\infty} |\varphi(t)|^2 \bigl(s\sinh(t)+\cosh(t)\bigr)\dd t},$$
for $s\in(0,\infty)$. The latter property follows by computing the first derivative of $g$
and arguing exactly as in Steps~3 and 4 of Proposition~\ref{mpeig}. Finally, we employed the density
result of Lemma~\ref{lem:1ddensity}.
\end{proof}
\begin{Corollary}
  Let $\Omega\subset\mathbb{H}^2$ be an open, bounded and geodesically convex set. If
    $$ \frac{|\Omega| + 2\pi}{\mathcal{H}^{1}(\partial\Omega)}\leq 1,$$
then, for any $\alpha<\aa_\star(\Omega^\ext)$ and $R>0$, it holds
$$\lambda_1^{\alpha}(\Omega^\ext)<\lambda_1^{\alpha}(B_R^\ext).$$
\end{Corollary}
\begin{Corollary}
   Given $\alpha\in\mathbb{R}$, then
   $$\inf_{R>0}\lambda_1^{\alpha}(B_R^\ext)>-\infty.$$ 
\end{Corollary}
In the context of Faber-Krahn type inequalities, it is common to make comparisons using either the disk of equal measure or the disk of equal perimeter as extremal sets. However, our approach deviates from this traditional choice, as we adopt  a different comparison that allows for a sharper estimate. The isoperimetric and the isochoric spectral inequalities follow as a consequence
of our stronger result.
\begin{Corollary}
Let $\Omega\subset\mathbb{H}^2$ be an open, bounded and geodesically convex set. Let $B\subset\dH^2$ be a geodesic disk  such that either $\cH^1(\p\Omg) = \cH^1(\p B)$ or $|\Omg| = |B|$, then it holds
$$\lambda_1^{\alpha}(\Omega^\ext)\leq \lambda_1^{\alpha}(B^\ext),\qquad \text{for all}\,\,\,\aa\in\dR.$$
\end{Corollary}
\begin{proof}
 Denoting by $B_{R^\star}$, $B_{R^\sharp}$ the geodesic disks having  respectively the same perimeter and area of $\Omega$, the geometric isoperimetric inequality ensures that $R^\sharp\leq R^\star$. Using the isoperimetric inequality in the hyperbolic plane again, we get
 \begin{align*}
       \frac{|\Omega| + 2\pi}{\mathcal{H}^{1}(\partial\Omega)}\leq    \frac{|B_{R^\star}| + 2\pi}{\mathcal{H}^{1}(\partial B_{R^\star})}=\coth(R^\star)\leq \coth(R^\sharp).  
  \end{align*} 
  Therefore, Theorem \ref{thm:main} implies
  $$\lambda_1^{\alpha}(\Omega^\ext)\leq \lambda_1^{\alpha}(B_{R^\star}^{\text{ext}})\leq \lambda_1^{\alpha}(B_{R^\sharp}^{\text{ext}}).\qedhere$$
\end{proof}

Another interesting optimisation result concerning the critical coupling parameter follows immediately from Theorem \ref{thm:main}.
\begin{Corollary}
    Let $\Omega\subset\mathbb{H}^2$ be an open, bounded and geodesically convex set. Denote by $B_R\subset\dH^2$ a geodesic disk of radius $R$ such that $$\coth(R)\geq\frac{|\Omega| + 2\pi}{\mathcal{H}^{1}(\partial\Omega)},$$
then it holds
$$\alpha_{\star}(\Omega^\ext)\geq \alpha_{\star}(B_R^\ext).$$
\end{Corollary}
\section*{Acknowledgments}
A.C. was partially supported by the INdAM-GNAMPA project 2025
``Propriet\`{a} qualitative e regolarizzanti di equazioni ellittiche e
paraboliche'', cod. CUP E5324001950001.
D.K. was supported by the grant no. 26-21940S
of the Czech Science Foundation.

\begin{appendix}
	\section{Convexity of sets in Riemannian manifolds}
	\label{app:convexity}
	In this appendix, we summarise some general notions of convexity in Riemannian manifolds introduced in \cite{CG72, SS89}, see also~\cite{A78} for further details, and provide some of their properties.
  	\begin{definition}\label{def:convexity}
		Let $M$ be a Riemannian manifold and let $C \subset M$. We say that:
		\begin{enumerate}[label=(\alph*)]
			\item $C$ is \emph{weakly convex} if for every $p, q \in C$, there exists a minimal geodesic $\gamma: [a,b] \to M$ connecting $p$ and $q$ such that $\gamma([a,b]) \subseteq C$;
			\item $C$ is \emph{strongly convex} if for every $p, q \in C$ there exists a unique minimal geodesic $\gamma: [a,b] \to M$ connecting $p$ and $q$ in $M$, and $\gamma([a,b]) \subseteq C$;
			\item $C$ is \emph{locally convex} if for every $p \in \overline{C}$ there exists $\varepsilon > 0$ such that $C \cap B_\varepsilon(p)$ is strongly convex, where $B_\varepsilon(p)$ is the metric ball of radius $\varepsilon$ centred at $p$;
            \item Suppose further that $C$ is a domain with $C^2$-smooth boundary. Then $C$ is \emph{strictly convex} if the second fundamental form $\mathrm{II}$ of $\partial C$ with respect to the outward unit normal $\nu$ is positive definite at every point $p \in \partial C$:
            $$\mathrm{II}_p(\xi, \xi) = \langle \nabla_\xi \nu, \xi \rangle > 0, \quad \forall \xi \in T_p(\partial C) \setminus \{0\}.$$
            
		\end{enumerate}
	\end{definition}
 We also recall the notion of supporting element (see \cite{CG72, A78}). 
	\begin{definition}
		Let $M$ be a Riemannian manifold, and let $C\subseteq M$ be an open set. For any point $p\in\partial C$ and a tangent vector $\nu\in T_pM$, 
		consider the half-space
		\[
		H_p=\Set{\xi \in T_p M \,\colon\, \scalar{\nu}{\xi}< 0}.
		\]
		\begin{enumerate}[label=(\roman*)]
			\item  $H_p$ is a \emph{supporting element} for $C$ in $p$, if for every $q$ in the interior of $C$ and for every minimal geodesic 
			\[
			\gamma:[0,1]\to M
			\]
			satisfying $\gamma(0)=p$ and $\gamma(1)=q$, it holds that $\gamma'(0)\in H_p$;
			\item $H_p$ is a \emph{locally supporting element} for $C$ in $p$ if there exists a neighbourhood $U$ of $p$ such that $H_p$ is a supporting element for $U\cap C$ in $p$.
		\end{enumerate}
	\end{definition}
\noindent	
	We have the following characterisation of weak convexity, originally proved in \cite[Proposition~2]{A78}. We recall that for every point $p\in M$, $\cut(p)$ denotes the cut-locus of $p$. 
	\begin{Proposition}
		\label{prop: weakchar}
		Let $C \subset M$ be a connected open subset of a Riemannian manifold $M$. Then $C$ is weakly convex if and only if for every point $p \in \partial C$ there exists a locally supporting element at $p$, and the set $C \setminus \cut(p)$ is connected.
	\end{Proposition}
	\noindent We also have that a locally supporting element always exists for open, locally convex sets. Indeed, Cheeger and Gromoll in \cite[Theorem 1.6, Lemma 1.7]{CG72} proved the following result.
	\begin{Theorem}
		\label{teor: supporting}
		Let $M$ be a Riemannian manifold of dimension $n$, and let $C\subseteq M$ be a non-empty, open, locally convex set. Then $\partial C$ is an embedded $(n-1)$-dimensional topological submanifold of $M$, and it has a supporting element in every point $p\in\partial C$.
	\end{Theorem}

	\noindent In the original work, the authors considered closed sets; however, it is evident that if $C$ is locally convex, then its closure  $\overline{C}$ is also locally convex, and moreover, the boundary satisfies $ \partial C = \partial \overline{C} $. Additionally, by definition, any supporting element for $ \overline{C} $ is also a supporting element for $ C $.
	
	We provide an approximation result proved by Bangert in \cite[Corollary 2.5, Corollary 2.6]{B78}.
	\begin{Theorem}
		\label{teor: approx}
		Let $M$ be a Riemannian manifold, and let $C\subset M$ be a connected, compact, locally convex set with non-empty interior. Assuming either that
		\begin{enumerate}[label=(\alph*)]
			\item the sectional curvatures are positive on $C$;
			\item the sectional curvatures are negative on $C$;
		\end{enumerate}
		then there exists a sequence of connected, compact, locally convex sets $\{C_k\}_{k\in\dN}$ with $C^\infty$ boundaries such that
		\[
		\lim_{k\to\infty}\big( d^H(C_k,C)+d^H(\partial C_k,\partial C)\big)=0.
		\]
	\end{Theorem}
\noindent We also recall the following property of locally convex sets in Riemannian manifolds proved in \cite[Theorem 6.1]{W76}.
\begin{Theorem}
	\label{teor: regularity}
	Let $M$ be a Riemannian manifold, and let $C\subset M$ be a closed, locally convex set. Then $C$ has strongly Lipschitz boundary.
\end{Theorem}

\section{Steiner's formula}
\label{app:steiner}
Given a Riemannian manifold $M$, for each point $p \in M$ and $r > 0$, we denote by $B_r(p)$ the metric ball centred at $p$ with radius $r$. For sufficiently small $r$, this metric ball coincides with the geodesic ball $\exp_p(\cB_r(0))$, where $\cB_r(0)$ is the metric ball in the tangent plane $T_pM$. Following~\cite{K91}, we introduce below the concepts of metric projection and the reach.
\begin{definition}
	Let $M$ be a Riemannian manifold and let $\Omega \subset M$ be a non-empty set. For any point $p \in M$, a \emph{metric projection} of $p$ onto $\Omega$ is a point $q \in \Omega$ such that
	\[
	d(p, q) = \inf_{x \in \Omega} d(p, x).
	\]
	When this point is unique, we denote it by $\sigma_{\Omega}(p)$.
\end{definition}
\begin{definition}
	For a given $q \in \Omega$, the \emph{reach} of $q$ with respect to $\Omega$ is defined as the supremum of radii $r > 0$ such that every point $p$ in the ball $B_r(q)$ admits a unique metric projection onto $\Omega$:
	\[
	\mathcal{R}(q) = \sup\{ r > 0 \,\colon\, \forall p \in B_r(q), \text{ the metric projection of } p \text{ onto } \Omega \text{ is unique} \}.
	\]
	We then define the \emph{reach} of the set $\Omega$ as
	\[
	\mathcal{R}(\Omega) = \inf_{q \in \Omega} \mathcal{R}(q).
	\]
	The set $\Omega$ is said to have \emph{positive reach} if $\mathcal{R}(\Omega) > 0$.
\end{definition}
\begin{Remark}
    A set $\Omg\subset M$ with positive reach is closed. Suppose, by contradiction, that there exists $x\in \partial\Omega\setminus\Omega$, and let
$y$ be its unique metric projection onto $\Omega$. Then any disk centred at $x$ with radius $R<d(x,y)$ does not intersect $\Omega$, which leads to an absurd.
\end{Remark}

It is a well-known fact that in $\mathbb{R}^n$, any convex set $C$ has infinite reach, i.e.,
\[
\mathcal{R}(C) = +\infty.
\]
\noindent When considering simply-connected space forms, related but nuanced results hold for connected, compact, locally convex sets. In particular, the following proposition \cite[Lemma 2.2]{K94}, provides useful lower bounds for the reach in this setting.
\begin{Proposition}
	\label{lowboundreach}
	Let $M$ be a simply-connected space form with constant curvature $\kappa$, and let $C \subseteq M$ be a connected, compact, locally convex set. We have:
	\begin{enumerate}[label=(\roman*)]
		\item If $\kappa < 0$, then
		\[
		\mathcal{R}(C) = +\infty;
		\]
		\item If $\kappa > 0$, then
		\[
		\mathcal{R}(C) \geq \frac{\pi}{2} \frac{1}{\sqrt{\kappa}}.
		\]
	\end{enumerate}
\end{Proposition}
\noindent Let $M$ be a complete Riemannian $n$-manifold with constant sectional curvature $\kappa$. We introduce the functions 
\[
\sn_\kappa(t)=\begin{dcases}
	\frac{1}{\sqrt{-\kappa}}\sinh(\sqrt{-\kappa}t) &\text{if }\kappa<0, \\
	t &\text{if }\kappa=0 \\
	\frac{1}{\sqrt{\kappa}}\sin(\sqrt{\kappa}t) &\text{if }\kappa>0, \\
\end{dcases}
\]
and $\cn_\kappa$=$\sn_\kappa'$. For $1\le  j\le n$, we define
\[
L_j(t):=\int_0^t \cn_\kappa^{n-j}(t)\sn_\kappa^{j-1}(t)\,dt,
\]
and
\[
L_0(t)=1.
\]
The following result was proved by Kohlmann in \cite[Theorem 2.7]{K91}.
\begin{Theorem}[Steiner's formula on simply-connected space forms]
	\label{Theorem: steiner}
	Let $M$ be a simply-connected space form of dimension $n$ and curvature $\kappa$, and let $\Omega\subset M$ be a set of positive reach. Let $U$ be a suitable open neighbourhood of $\Omega$ in which the metric projection $\sigma_\Omega$ is well defined. For every $j=0,\dots,n$ there exist coefficients $\Phi_j(\Omega)$ such that the following hold: if $s>0$ is such that 
	$\overline{\Omega_s}\subset U$, then we have
	\[
	\mathcal{H}^{n-1}(\partial\Omega_s)=\sum_{r=0}^{n-1} \cn_\kappa^{r}(s)\sn_\kappa^{n-1-r}(s) \Phi_r(\Omega),
	\qquad\text{and}\qquad
	\abs{\Omega_s}=\sum_{r=0}^n L_{n-r}(s)\Phi_{r}(\Omega).
	\]
	In particular, 
	\[
	\Phi_n(\Omega)=\abs{\Omega}. 
	\]
	Moreover, if $\partial\Omega$ is a $C^2$-smooth, compact, embedded $(n-1)$-submanifold of $M$, then for every $r=0,\dots, n-1$ we have
	\begin{equation}
		\label{eq: curvaturemeasure}
		\Phi_r(\Omega)=\binom{n-1}{r}\int_{\partial\Omega} H_{n-1-r}(p)\dd\mathcal{H}^{n-1}(p),
	\end{equation}
	where $H_j$ denote the normalised $j$-th homogeneous symmetric function of the principal curvatures of $\partial \Omega$.
\end{Theorem}
For sets with $C^2$-smooth boundary, the formulation of the curvature measures $\Phi_r$ provided in \cite[Theorem 2.7]{K91} is equivalent to \eqref{eq: curvaturemeasure}, as shown in \cite[Remark 2.44]{ACCNT}. For an explicit definition of the functions $H_j$, we refer to \cite[Equation 2.13]{K91}.
\begin{Remark} Observe that when $\kappa = 0$, we get the Steiner polynomial
	\[
	\abs{\Omega_s}=\abs{\Omega }+\sum_{k=1}^{n}\frac{s^k}{k} \Phi_{n-k}(\Omg).
	\]
\end{Remark}

\begin{Remark}\label{per}
	Let $M$ be a simply-connected space form. If $\Omega \subset M$ is a set of positive reach with strongly Lipschitz boundary $\partial \Omega$, then it holds that
	\[
	\Phi_{n-1}(\Omega) = \mathcal{H}^{n-1}(\partial \Omega).
	\]
	In the Euclidean case $M = \mathbb{R}^n$ with the standard metric, this equality is classical (see, e.g., \cite[Theorem 2.106]{AFP00}). For a general simply-connected space form, the equality follows by passing to normal coordinates and applying the Euclidean result locally (cf. \cite[Remark 2.45]{ACCNT}). Moreover, an application of Steiner's Formula shows that 
    $$
        \Phi_{n-1}(\Omega)=\lim_{s\to 0^+}\frac{\abs{\Omega_{s}}-\abs{\Omega}}{s}.
    $$
    The right-hand side is the $(n-1)$-dimensional Minkowski content of $\Omega$ \cite[Definition 1.51]{R23}. If $\Omega$ has finite measure, using \cite[Theorem 1.52]{R23}, we get
    $$P(\Omega)=\lim_{s\to 0^+}\frac{\abs{\Omega_{s}}-\abs{\Omega}}{s}= \Phi_{n-1}(\Omega)=\mathcal{H}^{n-1}(\partial \Omega),$$
    where $P(\Omega)$ denotes the perimeter in the sense of De Giorgi \cite[Definition 1.37]{R23}.
  \end{Remark}

  \noindent A continuity property for the curvature measures $\Phi_j$ has been established in \cite[Theorem 2.4]{K94}.
\begin{Theorem}
	\label{teor: curvconv}
	Let $M$ be a simply-connected space form of dimension $n$ and curvature $\kappa$. Let $\{\Omega_k\}_{k\in\mathbb{N}} \subset M$ be a sequence of compact sets with non-empty boundaries. Suppose there exists $\delta > 0$ and a compact set $\Omega \subset M$ such that  
	\[
	\mathcal{R}(\Omega_k) \geq \delta \quad \text{and} \quad \lim_{k \to \infty} d^H(\Omega_k, \Omega) = 0.
	\]  
	Then, for every $r = 0, \dots, n$,  
	\[
	\Phi_r(\Omega_k) \to \Phi_r(\Omega),\qquad\text{as}\,\,k\arr\infty.
	\]  

\end{Theorem}

\section{Auxiliary properties of Sobolev spaces}\label{app:Sobolev}
In this appendix, we will provide some fine properties of 
$W^{1,2}$-functions on exterior domains in the hyperbolic plane.
\begin{Lemma}\label{lem:Ehrling}
	Let $\Omg\subset\dH^2$ be a bounded, connected, Lipschitz domain. Then for any $\eps > 0$, there exists a constant $C(\eps) > 0$ such that the inequality
	\[
		\int_{\p\Omg^\ext}|u|^2\dd \s 
		\le 	
		\eps\int_{\Omg^\ext}|\nb u|^2\dd \mu 
		+ 
		C(\eps)\int_{\Omg^\ext}|u|^2\dd\mu
	\]
	holds for all $u\in W^{1,2}(\Omg^\ext)$.
\end{Lemma}
\begin{proof}
	We divide the proof into three steps.
	
	\smallskip
	
	\noindent\textit{Step 1: reduction to Sobolev functions with zero mean value.} We begin by proving the inequality for a bounded, Lipschitz, and connected set $U\subset\dH^2$. In this framework, following the reasoning in \cite[Section II, §4]{CH}, we may assume without loss of generality that
	the function $u\in W^{1,2}(U)$ has zero mean value on $U$. Indeed, let 
	$$
	\langle u\rangle_U = \frac{1}{|U|}\int_U u\dd\mu, \qquad 
	v = u - \langle u\rangle_U.
	$$
	Then $v$ has zero mean and $\nabla v = \nabla u$.  
	By the reverse triangle inequality,
	$$
	\|v\|_{L^2(\partial U)}^2 
	\ge \frac12\|u\|_{L^2(\partial U)}^2 - \big|\langle u\rangle_U\big|^2\,\mathcal{H}^1(\partial U).
	$$
	Applying the desired inequality to $v$, yields 
	$$
	\|v\|_{L^2(\partial U)}^2 
	\le \varepsilon \|\nabla v\|_{L^2(U)}^2 
	+ C(\eps) \|v\|_{L^2(U)}^2
	= \varepsilon \|\nabla u\|_{L^2(U)}^2 
	+C(\eps) \big\|u - \langle u\rangle_U\big\|_{L^2(U)}^2.
	$$
	Next we estimate the mean term.  
	By the Cauchy--Schwarz inequality,
	$$
	|\langle u\rangle_U|^2
	= \frac{1}{|U|^2}\left|\int_U u\dd\mu\right|^2
	\le \frac{1}{|U|}\|u\|_{L^2(U)}^2,
	$$
	and
	$$
	\|u - \langle u\rangle_U\|_{L^2(U)}^2
	\le 2\|u\|_{L^2(U)}^2 + 2|\langle u\rangle_U|^2\,|U|
	\le 4\|u\|_{L^2(U)}^2.
	$$
	Combining the previous estimates yields
	$$
	\|u\|_{L^2(\partial U)}^2
	\le 2\varepsilon\|\nabla u\|_{L^2(U)}^2
	+ \left(8C(\eps)+
	 \frac{2\mathcal{H}^1(\partial U)}{|U|}\right)\|u\|_{L^2(U)}^2.
	$$
	Finally, for any $\eps > 0$ there exists $C'(\eps) > 0$, depending only on $\eps$ and $U$, such that 
	\begin{align}\label{ergen}
		\|u\|_{L^2(\partial U)}^2
		\le \varepsilon \|\nabla u\|_{L^2(U)}^2
		+ C'(\eps)\|u\|_{L^2(U)}^2. 
	\end{align}
holds for all $u\in W^{1,2}(U)$.
	
	\smallskip

	\noindent\textit{Step 2: proof of the inequality for a bounded, connected, Lipschitz set.} We introduce the subspace of $W^{1,2}(U)$ consisting of functions with zero mean value, denoted by $$\tilde{W}^{1,2}(U):=\left\{u\in W^{1,2}(U)\,\colon\, \langle u \rangle_U=0 \right\}.$$ 
	The Poincaré inequality \cite[Corollary A.1.2]{JJ} ensures that the seminorm defined by the $L^2$-norm of the gradient is a norm on $\tilde{W}^{1,2}(U)$, equivalent to the standard $W^{1,2}$-norm. Since the trace operator $\text{Tr}:W^{1,2}(U)\to L^2(\partial U)$ is continuous \cite[Theorem 7.5]{BM13}, it follows from \cite[Theorem 2.1]{Oku} that for any $\eps>0$, there exists a constant $C(\eps)>0$ such that, for every $u\in\tilde{W}^{1,2}(U)$, inequality \eqref{ergen} holds.
	
	\smallskip
	
	\noindent\textit{Step 3: proof of the inequality for the exterior of a bounded, Lipschitz and connected set.}
	We fix an open ball $B$ containing $\ov{U}$ and define $V:=B\sm\ov{U}$. Given $u\in W^{1,2}(U^{\text{ext}})$ then
	$$ \|u\|_{L^2(\partial U)}^2\le \|u\|_{L^2(\partial V)}^2
	\le \varepsilon \|\nabla u\|_{L^2(V)}^2
	+ C(\eps)\|u\|_{L^2(V)}^2\leq \varepsilon \|\nabla u\|_{L^2(U^{\text{ext}})}^2
	+ C(\eps)\|u\|_{L^2(U^{\text{ext}})}^2, $$
	by which the lemma is proved.
\end{proof}
We also prove the following useful approximation result, which extends~\cite[Theorem 2.9]{AUB} to exterior domains in the hyperbolic plane.
	\begin{Lemma}\label{lem:density}
		Let $k\in\dN$ and assume that $\Omega\subset\mathbb{H}^2$ is an open, bounded set with $C^k$-smooth boundary. Then $C^k_c(\overline{\Omega^\ext})$ is dense in $W^{1,2}(\Omega^\ext)$.
	\end{Lemma}
	\begin{proof}
		We observe that, according to \cite[Theorem 2.9]{GBP}, the subspace $C^\infty(\Omg^\ext) \cap W^{1,2}(\Omg^\ext)$ is dense in $W^{1,2}(\Omg^\ext)$.
		We divide the proof into two steps. 
		
		\noindent $\textit{Step 1.}$ We first show the density in $W^{1,2}(\Omg^\ext)$ of functions in $C^\infty(\Omg^\ext) \cap W^{1,2}(\Omg^\ext)$ that vanish outside a geodesic disk containing $\Omega$. Clearly, such functions are, in general, not smooth up to $\partial \Omega$.
		
		Let $f
		\in C^\infty(\dR)$ be a decreasing function such that $f(t) = 1$ for $t \le 0$
		and $f(t) = 0$ for $t \ge 1$. 
		For a fixed point $y\in\Omega$, the triangle inequality ensures that the distance function $d_{\mathbb{H}^2}(x,y)$ is $1$-Lipschitz. Moreover, from \eqref{dh2}, we get that $d_{\dH^2}(\cdot,y)\in C^{\infty}(\Omega^\ext)$. Given $\varphi \in C^\infty(\Omega^\ext) \cap W^{1,2}(\Omg^\ext)$, we claim that the sequence of functions 
		\begin{align}\label{app}
			\varphi_j(x) = \varphi(x) f(d_{\mathbb{H}^2}(x,y) - j),
			\qquad j\in\dN,
		\end{align}
		converges to $\varphi(x)$ in $W^{1,2}(\Omg^\ext)$.
		When $j \to \infty$, $\varphi_j(x) \to \varphi(x)$ for all $x\in\Omg^\ext$. Moreover, 
		$|\varphi_j| \le |\varphi|$ for all $j\in\dN$ and $|\varphi|$ belongs to $L^2(\Omega^\ext)$. So by the Lebesgue dominated 
		convergence theorem,
		\[
		\|\varphi_j - \varphi\|_{L^2(\Omega^\ext)} \to 0.
		\]
		Moreover, when $j \to \infty$, $|\nabla \varphi_j(x)| \to |\nabla \varphi(x)|$ for all $x\in\Omg^\ext$, and
		\[
		|\nabla \varphi_j(x)| \le |\nabla \varphi(x)| + |\varphi(x)| \sup_{t \in [0,1]} |f'(t)|,\qquad \text{for all}\,j\in\dN		\]
		where the right-hand side belongs to $L^2(\Omg^\ext)$. Thus
		\[
		\|\nabla(\varphi_j - \varphi)\|_{L^2(\Omg^\ext)} \to 0.
		\]
		\textit{Step 2.} 
		To complete the proof, it suffices to show that any given function $\psi\in C^\infty(\Omg^\ext) \cap W^{1,2}(\Omg^\ext)$ that vanishes outside a geodesic disk containing $\Omega$ can be approximated in $W^{1,2}(\Omg^\ext)$ by functions in $C^k_c(\overline{\Omg^\ext})$.
		
		Let $B$ be a geodesic disk such that $B\supset\Omega$, $B\supset\supp\psi$ and $d_{\mathbb{H}^2}(\partial B, \text{supp}(\psi))>0$. Since $\overline{B\setminus\Omega}$ is a compact manifold with $C^k$-smooth boundary, it admits a finite $C^k$-atlas $\{(\Omega_i, \varphi_i)\}^r_{i=1}$,  each $\Omega_i$ is homeomorphic either to a disk in $ \mathbb{R}^2$, or, if it intersects the boundary, to a half-disk in $\dR^2$. We choose the domains $\Omega_j$ intersecting $\partial B$ in such a way that they are disjoint with $\text{supp}(\psi)$. Let $\{\alpha_i\}$ be a $C^\infty$ partition of unity subordinate to the cover $\{\Omega_i\}$ of $\overline{B\setminus\Omega}$. By definition, $\text{supp}(\alpha_i)\subset\Omega_i$. Since each $\Omega_i$ is open in the induced topology, it follows that $\text{supp}(\alpha_i)\cap\partial\Omega_i\subseteq \partial\Omega_i\cap (\partial\Omega\cup \partial B)$. We have to approximate each $\alpha_i \psi$ in the Sobolev norm by functions in $C^k(\overline{B\setminus\Omega})$ which vanish on $\partial B$.
		
		If $\Omega_i$ is
		homeomorphic to a disk or to a half–disk with $\Omega_i\cap\partial B\not=\emptyset$, then the function $\alpha_i \psi$, which is identically zero in the latter case, does not need to be approximated. Observe that in both cases $\alpha_i \psi\in C^\infty_c(\Omega_i)$ and $\alpha_i \psi=0$ on $\partial\Omega_i$.
		
		If $\Omg_i$ is homeomorphic to a half-disk $D$ with $\Omega_i\cap\partial \Omega\not=\emptyset$, we can proceed as in the proof of~\cite[Theorem 2.9]{AUB}. We introduce the half-space $E:=\{x\in\mathbb{R}^2\,\colon\,x_1\leq 0\}$ and without loss of generality, we assume $D=\mathring{B_r}(x)\cap E$ with $x\in \partial E$. The sequence of functions $h_m$, defined for $m\in\dN$ sufficiently large, as the restriction to $D$ of
        $$\big(\alpha_i \psi\circ \varphi_i^{-1}\big)\big(x_1 - \tfrac1m, x_2\big),\quad (x_1,x_2)\in \mathring{E},$$
        converges to $\alpha_i \psi\circ \varphi_i^{-1}$ in $W^{1,2}(D)$. Clearly, $h_m \circ \varphi_i \in C^k(\Omega_i)$. Since the metric tensor is bounded on $\Omega_i$, expressing the $W^{1,2}(\Omega_i)$-norm in local coordinates via the chart $\varphi_i$ shows that it is equivalent to the norm in $W^{1,2}(D)$. Therefore $h_m \circ \varphi_i$ converges to $ \alpha_i \psi $ in $ W^{1,2}(\Omega_i)$. Observe that $\alpha_i \psi=0$ on $\partial\Omega_i\setminus\partial \Omega$. Moreover, $\varphi_i$ maps interior points of $\Omega_i$ to interior points of $D$ and points of $\partial \Omega_i\cap\Omega$ to points of $D\cap \partial E$ (see e.g. \cite[Theorem 1.37]{LSM}).  As a consequence,  $\text{supp}(h_m)\cap(\partial D\setminus\partial E)=\emptyset$ and therefore $\text{supp}(h_m \circ \varphi_i)\cap(\partial\Omega_i\setminus\partial \Omega)=\emptyset$. Given $\eps>0$, if  $\Omega_i\cap \partial\Omega=\emptyset$, we set $\eta_i:=\alpha_i\psi$; otherwise, we define $\eta_i:=h_{m_i}\circ \varphi_i$ where $m_i$ is chosen so that $\|h_{m_i}\circ \varphi_i-\alpha_i\psi\|_{W^{1,2}(\Omega_i)}<\eps.$ Defining
        $$\eta:=\sum_{i=1}^r\eta_i,$$
        we obtain
        \begin{align*}
            \|\eta-\psi\|_{W^{1,2}(B\setminus\Omega)}&=\left\|\sum_{i=1}^r\eta_i-\sum_{i=1}^r\alpha_i\psi\right\|_{W^{1,2}(B\setminus\Omega)}\\&\leq \sum_{i=1}^r\|\eta_i-\alpha_i\psi\|_{W^{1,2}(B\setminus\Omega)}\\&=\sum_{i=1}^r\|\eta_i-\alpha_i\psi\|_{W^{1,2}(\Omega_i)}\\&\leq r\eps.
        \end{align*}
        From the construction, it is clear that $\eta\in C^k(\overline{B\setminus\Omega})$ and $\eta=0$ on $\partial B$.
   \end{proof}

\begin{Lemma}\label{lem:1ddensity}
	For any $R > 0$, the space $C^\infty_0([0,\infty))$ is dense in $W^{1,2}(\dR_+;\sinh (R+t) \dd t)$.
\end{Lemma}	
\begin{proof}
	Let $\chi\in C^\infty([0,\infty))$ be such that $\chi(t) = 1$ for all $t\in[0,1]$, $\chi(t) = 0$ for all $t \ge 2$ and $0\le \chi(t) \le 1$ for all $t\in\dR_+$.
	Let $\psi\in W^{1,2}(\dR_+;\sinh (R+t) \dd t)$ be arbitrary.
	Clearly, for any $n\in\dN$ we have $\chi(\frac{\cdot}{n})\psi \in W^{1,2}(\dR_+;\sinh (R+t) \dd t)$. Let us fix $\eps > 0$.
	Notice that there exists a sufficiently
	large $n\in\dN$ such that 
	\begin{equation}\label{eq:approx1}
		\left\|\psi - \chi\left(\tfrac{\cdot}{n}\right)\psi\right\|_{W^{1,2}(\dR_+;\sinh (R+t) \dd t)} < \eps.
	\end{equation}
	Next, we will use that the norms $W^{1,2}((0,2n))$ and $W^{1,2}((0,2n);\sinh(R+t)\dd t)$ are equivalent, because the weight and its inverse are both bounded on the interval $[0,2n]$. Thus, by  construction $\chi\left(\tfrac{\cdot}{n}\right)\psi\in 
	\{f\in W^{1,2}((0,2n)\,\colon\,f(2n) = 0\}$. Moreover, by standard
	properties of (non-weighted) Sobolev spaces on the interval,
	there exists a function $\phi\in C^\infty_0([0,2n))$ such that
	\begin{equation}\label{eq:approx2}
		\|\chi\left(\tfrac{\cdot}{n}\right)\psi -\phi\|_{W^{1,2}([0,2n];\sinh(R+t)\dd t)} < \eps,
	\end{equation}
	Combining~\eqref{eq:approx1} and~\eqref{eq:approx2} we get the desired density result. 
\end{proof}
\end{appendix}

\bibliographystyle{plain}

\end{document}